\documentclass[twoside,a4paper,reqno]{amsart}
\usepackage[top=3.75cm,bottom=3.75cm,left=3.25cm,right=3.25cm]{geometry}
\usepackage{graphicx}

\usepackage{fancyhdr}
\usepackage{tikz}

\usepackage{xcolor}

\usepackage[utf8]{inputenc}
\usepackage{tikz}
\usetikzlibrary{positioning}

\usepackage{amsmath}
\usepackage{amsfonts}
\usepackage{amssymb}
\usepackage{amsthm}
\usepackage{ dsfont }
\theoremstyle{remark}

\usepackage{amsmath}
\usepackage{amsfonts}
\usepackage{amssymb}
\usepackage{amsthm}
\usepackage{ dsfont }
\usepackage{ esint }
\usepackage{enumerate}

\usepackage[greek,main=english]{babel}
\usepackage[utf8]{inputenc}

\usepackage[LGR, T1]{fontenc}

\newtheorem{theorem}{Theorem}[section]
\newtheorem*{theorem*}{Theorem}
\newtheorem*{proposition*}{Proposition}
\newtheorem{lemma}[theorem]{Lemma}
\newtheorem{proposition}[theorem]{Proposition}

\theoremstyle{definition}
\newtheorem{definition}[theorem]{Definition}

\newcommand\numberthis{\addtocounter{equation}{1}\tag{\theequation}}

\renewcommand{\d}{\mathrm{d}}

\begin{document}

\title[Kinetic and diffusive behaviour for a Rayleigh gas]{Derivation of kinetic and diffusion equations from a hard-sphere Rayleigh gas using collision trees and semigroups}
\author{Karsten Matthies \and  Theodora Syntaka}
\address{University of Bath,
Department of Mathematical Sciences,
Bath,
BA2 7AY,
United Kingdom}
\email{k.matthies@bath.ac.uk; ts2291@bath.ac.uk}

\begin{abstract}
We will revisit the classical questions of understanding the statistics of various deterministic dynamics of $N$ hard spheres of diameter $\varepsilon$ with random initial data in the Boltzmann-Grad scaling as $\varepsilon$ tends to zero and $N$ tends to infinity.   The convergence of the empiric particle dynamics to the Boltzmann-type dynamics is shown using semigroup methods to describe probability measures on collision trees associated to physical trajectories in the case of a Rayleigh gas. As an application we derive the diffusion equation by a further rescaling.
\end{abstract}

\maketitle


\pagenumbering{arabic}


\section{Introduction}
The derivation of various continuum equations from a discrete deterministic system of particles is of major interest. This is an area of research in mathematical physics related to Hilbert's Sixth Problem. There are two major parts of this programme. Firstly, there is the direct derivation of a continuum equation from particle models as a scaling limit for many small particles. Secondly, these derived continuum equations can be rescaled again to derive further continuum equations as effective descriptions. In kinetic theory, this problem has been approached in two steps using the Boltzmann equation as the intermediate, mesoscopic description. The first one is to derive kinetic equations, such as the Boltzmann equation, from a system of particles and the second is to derive further continuum equations, such as the Navier-Stokes, Euler and Heat equations, from the Boltzmann equation.  For some review of the substantial literature see e.g. \cite{Slemrod2013, Gallagher2019}. The general idea can be summarised in the following figure.

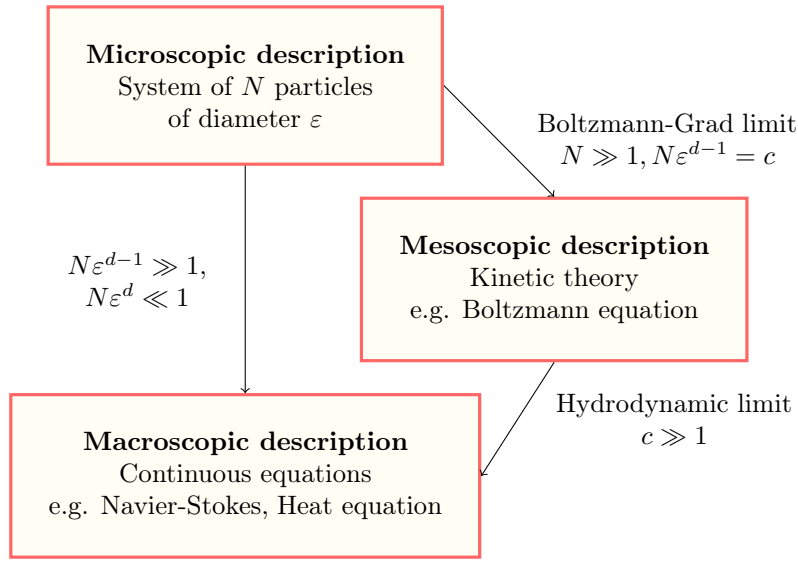
\begin{figure}[h!]

\begin{center}
\begin{tikzpicture}[
inner sep=5mm,
roundnode/.style={circle, draw=green!60, fill=green!5, very thick, minimum size=1cm},
squarednode/.style={rectangle, draw=red!60, fill=yellow!5, very thick, minimum size=1cm},
title/.style={font=\LARGE\scshape,node distance=16pt, text=black!40, inner sep=1cm},
]

 \node[]      (maintopic)                              {};
\node[squarednode,align = center]        (uppercircle)       [above=of maintopic] {\textbf{Microscopic description} \\
System of $N$ particles \\ of diameter $\varepsilon$ };
\node[squarednode,align = center ]      (rightsquare)       [right=of maintopic] {\textbf{Mesoscopic description} \\Kinetic theory \\ e.g. Boltzmann equation};
\node[ squarednode, align = center]        (lowercircle)       [below=of maintopic] {\textbf{Macroscopic description}\\Continuous equations \\ e.g. Navier-Stokes, Heat equation };

\draw[->, align = center] (uppercircle.south) -- node[left] {$N \varepsilon ^{d-1}\gg 1$,\\  $N \varepsilon ^{d}\ll 1$ } (lowercircle.north);
\draw[->, align = center] (uppercircle.east) --  node[right] { Boltzmann-Grad limit \\$N \gg 1, N \varepsilon ^{d-1} = c$} (rightsquare.north);
\draw[->,align = center] (rightsquare.south) -- node[right] { Hydrodynamic limit \\ $ c \gg1$}  (lowercircle.east);
\end{tikzpicture}
\end{center}

\end{figure}

The equations have very different dynamical behaviours on the different scales, e.g. the system of particles has a time-reversible dynamics  while  the kinetic equations are irreversible; also the phase spaces are substantially different such that the right way to compare them becomes part of the question of how to make sense of the diagram.

\subsection*{Particle dynamics}
We start by describing the particle dynamics. While Hamiltonian dynamics for a large number of particles $N$ can be considered, most research focuses on particle models with short range  interactions. The extreme case  are hard-sphere models, i.e.~solid particles of size $\varepsilon$ which undergo elastic collisions. A wider, relevant class are systems with short range potentials where the particles carry a force that affects only nearby particles, up to some distance proportional to $\varepsilon$. Our current work focuses on the hard-sphere model for mathematical and presentational simplicity. Suppose we have $N \in \mathds{N}$ 
particles identically and independently distributed that are moving in straight lines until they undergo a collision with another particle.
Then the equations of motion are given by
\begin{align*}
\frac{d x_i(t)}{dt} = v_i, \quad  \frac{d v_i(t)}{dt} = 0,
\end{align*}
provided that $|x_j(t) - x_i(t)| >\varepsilon$, where $x_i (t)$ is the position and $v_i(t)$ is the velocity of particle $i$, for $1\le i \neq j\le N$ and $t \ge 0$. Here, $\varepsilon$ is the diameter of the particles.

Else, there exists a $k$ such that $|x_i (t) - x_k(t)| = \varepsilon$. That is to say, particles $k$ and $j$ experience an elastic collision at time $t$. Then if $v_i( t^-)$ and $v_k( t^-)$ are the velocities before the collision and $v_i (t)$ and $v_k(t)$ the velocities after the collision we have
\begin{align}
v_i (t) &= v_i(t^-) - \nu \cdot (v_i(t^-) - v_k(t^-)) \nu, \nonumber \\
v_k (t) &= v_k(t^-) + \nu \cdot (v_i(t^-) - v_k(t^-)) \nu, \label{eqn:hardsphere}
\end{align}
where the collision parameter $\nu \in \mathbb{S}^2$ is denoted by $$\nu := \frac{ x_i(t) - x_k(t)}{| x_i(t) -x_k(t) |}.$$
These dynamics preserve overall kinetic energy and momentum.

\subsection*{The Boltzmann equation}

The Boltzmann equation is the paradigm of a kinetic equation. It describes the statistical behaviour of a thermodynamic system not in a state of equilibrium, that is, it describes the evolution of a distribution of an idealised dilute gas. The Boltzmann equation is given by
\begin{align}\label{eqn:1.1.}
\begin{cases}
\partial_t f_t(x,v) + v \cdot \nabla_x f_t(x,v) - \nabla_x F \cdot \nabla _v f_t &= c Q[f_t] (x,v), \\
 \quad \quad \quad \quad \quad \quad \quad \  \quad  \quad  \quad  \quad   \quad \ \ \ f_{t=0}(x,v) & = f_0(x,v),
\end{cases}
\end{align}
where $f: [0,T] \times \mathcal{U} \to \mathds{R}$ is the distribution of the gas at position $x$ and velocity $v$ at time $t$, with $\mathcal{U} = \mathds{T}^3 \times \mathds{R}^3$ be the phase space and $f_0$ is a given initial distribution. Here $c$ is a parameter which is the inverse of the mean free path of the microscopic particles and it represents the rate of collisions. The collision operator $Q$ represents the effect of interactions between the particles and $F$ is a potential of an external force. We will suppose that the potential $F = 0$.

The Boltzmann collision operator $Q[f_t]$ is quadratic in $f$ which means that is ruled by binary collisions and acts only in $v$. That is to say, the collisions are pointwise in $t$ and $x$. The collision operator is defined by
\begin{align*}
Q[f_t](x,v)=  \int_{\mathbb{S}^2} \int_{\mathds{R}^3} [  f_t(v') f_t(\bar{v}') - f_t(v) f_t(\bar{v}) ]  b(v-\bar{v}, \nu) \ \d\bar{v}\ \d{\nu},
\end{align*}
where the velocities $v'$ and $\bar{v}'$ are defined by $v'=v+\nu \cdot (\bar{v} -v)\nu$ and $\bar{v}'=\bar{v}-\nu \cdot (\bar{v} -v)\nu$. 
Here, $b$ is the collision cross-section which depends on modelling assumptions about the particle interaction, see \cite{cerci88}. For hard sphere dynamics, we have the positive part of the inner product:
\[b(v-\bar{v}, \nu) = [(v-\bar{v}) \cdot \nu] _+.  \]
One can show that ($v,\bar{v}, v', \bar{v}'$) provides the family of all solutions to the system of $3+1$ equations
\begin{align*}
v + \bar{v} &= v' + \bar{v}' \\
|v|^2 + |\bar{v}|^2 &= |v'|^2 + |\bar{v}'|^2, \numberthis \label{1.2.}
\end{align*}
which, at the kinetic level, express the fact that collisions are elastic and thus conserve momentum and energy.

There is a unique family of equilibrium distributions $M_{\beta}(v) := (\frac{\beta}{2\pi})^{\frac{3}{2}} \exp (- \frac{\beta}{2} |v|^2),$ for \eqref{1.1},  parameterised by the inverse temperature $\beta >0$, again see  \cite{cerci88} for details.

The rigorous justification of the full Boltzmann equation from a deterministic particle model of gas is a topic of current research.
The $f_0 \in L^1(\mathds{T}^3 \times \mathds{R}^3)$ is the initial distribution of the  particles. For the number of particles $N$ and their diameter we consider the Boltzmann-Grad scaling, $$N \varepsilon^2 = c.$$

We study the limiting behaviour of this model as the number of particles tends to infinity. As we increase the number of particles we decrease their radii such that the expected number of collisions in a given time remains constant. This is the Boltzmann-Grad scaling. The general  goal is to prove that   the probability of finding a particle at a given position converges to the solution of the Boltzmann equation as the number of particles tends to infinity.

The first result in this area was due to Lanford \cite{lanford75}. In this result, Lanford managed to prove convergence from a system of particles to the Boltzmann equation, in the case of hard spheres. King \cite{King1975} was able to prove the result of Lanford but with more general potential. This convergence is valid for short times, where the time of validity depends on the free flight time. For larger time interval convergence see \cite{Illner1986, Illner1989} who managed to obtain the convergence globally in time if the positions are in $\mathds{R}^d$ and the initial density is sufficiently small. The general results were substantially reworked in \cite{Gallagher2013, pul13}. Building on this, it is possible to understand fluctuations around the equilibrium solutions. Bodineau, Gallagher and Saint-Raymond in \cite{MR3455156} give a rigorous derivation of the Brownian motion as the limit of a deterministic system of hard-spheres, using the linear Boltzmann equation as a mesoscopic description. This result is valid for arbitrary large times and for initial distribution which is a perturbation of an equilibrium state with respect to the position of a tagged particle. The convergence rate of the distribution of the tagged particle to the solution of the linear Boltzmann equation is of the order of $(\log \log N)^{-1}$. An extension to derive the Ornstein-Uhlenbeck is given in \cite{Bodineau2018}. A substantial understanding of fluctuations around the equilibrium of the Boltzmann equation is provided in \cite{Bodineau2023a, Bodineau2023}. While there has been recent, substantial progress in the parallel problem of deriving the kinetic wave equation for the interaction of waves within nonlinear Schr\"odinger equations
\cite{Lukkarinen2011,Deng2021,Deng2023,Deng2023a}, the long-term derivation of the Boltzmann equation remains open.

\subsection*{Tagged particle models}

One option to simplify the dynamics --both on a particle level and on the level of the continuum equations--  is to split the system of particles into two kinds of particles, the background and the tagged particles. The tagged particle is single and interacts among a system of the background particles which are assumed not to interact among themselves.  This can be motivated e.g. for systems that are in equilibrium.

If the background particles are of infinite relative mass to the tagged particle and the background particles are fixed, then this model is known as the Lorentz gas model. Long-term convergence was shown in \cite{spohn78}. For recent progress  about random Lorentz gas see the paper of Lutsko and Toth, \cite{Lutsko2020}. In this work, they prove the invariance principle for a random Lorentz-gas particle in three space dimensions under the Boltzmann-Grad limit and simultaneous diffusion scaling. They do this by using a coupling of the mechanical trajectory and some controls on the efficiency of this coupling.

A Rayleigh gas is related with the Lorentz gas, but the background particles are no longer of infinite mass. If the background particles are of equal mass with the tagged particle and the background particles interact only with the tagged particle and not with each other,  then this model is known as the Rayleigh gas. We consider a tagged particle with initial distribution $f_0$ and background particles that are distributed according to $g_0$. The interaction between the tagged particle and the background follows \eqref{eqn:hardsphere}. There is no interaction among the background particles.
In an appropriate scaling limit this leads to a linear Boltzmann equation.
For hard-sphere particle dynamics, it is given by
\begin{align}\label{1.1}
\begin{cases}
\partial_t f_t(x,v) + v \cdot \nabla_x f_t(x,v) &= c Q[f_t] (x,v), \\
 \quad  \quad  \quad  \quad  \quad   \quad \ \ \ f_{t=0}(x,v) & = f_0(x,v).
\end{cases}
\end{align}
where $c$ is a parameter which is the inverse of the mean free path of the microscopic particles and it represents the rate of collisions.
The collision operator $Q$ is a linear operator describes the interactions of the particles with the surrounding medium and it is defined by $Q:= Q^+ - Q^-$, where the gain term $Q^+$ is given by
\begin{align*}
Q^+[f_t](x,v)=  \int_{\mathbb{S}^2} \int_{\mathds{R}^3}   f_t(x,v') g_0(\bar{v}') [(v-\bar{v})\cdot \nu]_+ \ \d \bar{v}  \ \d {\nu},
\end{align*}
where the pre-collisional velocities $v'$ and $\bar{v}'$ are given by $v'=v+\nu \cdot (\bar{v} -v)\nu$ and $\bar{v}'=\bar{v}-\nu \cdot (\bar{v} -v)\nu$ respectively, $[y]_+ := \max \{ y , 0\}$ and the loss term $Q^-$ is given by
\begin{align*}
Q^-[f_t](x,v)=  f_t(x,v) \int_{\mathbb{S}^2} \int_{\mathds{R}^3} g_0(\bar{v}) [(v-\bar{v})\cdot \nu]_+ \ \d \bar{v}\  \d {\nu}.
\end{align*}
The derivation of \eqref{1.1} is given in \cite{Matthies2018} for finite, fixed times without any error estimates. For variants and further details see also \cite{Stone2017, Matthies2018a} and for a related model \cite{Nota2019}. Various scaling limits can be considered for systems with long-range potential \cite{Nota2021}.

In \cite{fougères2024derivation}, Foug\'eres also considers a Rayleigh gas and shows quantitative rates of convergence obtaining better rates of convergence than in the full case in \cite{MR3455156}, which is of the order $ \exp ( -c_{\beta}  | \log N|^{1-\alpha} )$, $\forall \alpha >0$, with $N$ be the number of particles for times of the order $(\log|c_{\beta} \log \varepsilon| )^{\frac{1}{2}-\alpha}$ based on a careful analysis of the BBGKY hierarchy. Using different methods, we show in the present paper, that the convergence rate of the distribution of the tagged particle to the solution of the linear Boltzmann equation is of the order $\varepsilon^{\alpha}$, with $\alpha \in (0,\frac{11}{52})$ for time scales that are proportional to some negative power of $\varepsilon$.
\subsection*{Statement of main results}

In this paper we are providing two extensions to \cite{Matthies2018}: We are extending the time-scale of the derivation with quantitative error estimates, and secondly, we describe the long-term diffusive behaviour.

We are going to prove the following two theorems. The first one is a theorem for the derivation of the linear Boltzmann equation from a Rayleigh gas particle system with quantitative estimates on  a diverging time scale.

\definition The probability densities $f_0 \in L^1 (\mathds{T}^3 \times \mathds{R}^3)$ and $g_0 \in L^1(\mathds{R}^3)$ are admissible if
\begin{align*}
\int_{\mathds{T}^3 \times \mathds{R}^3} f_0(x,v) (1+|v|^2) \ \d x\,  \d v < \infty,
\int_{\mathds{R}^3} g_0(v) (1+|v|^3) \ \d v < \infty,
 \operatorname*{ess\,sup}_{v \in \mathds{R}^3} g_0(v) (1+|v|^5) &< \infty.
\end{align*}

\begin{theorem}\label{thm:1.1.} Let $f_0$ the initial distribution of the tagged particle and $g_0$ the distribution of the background particles be admissible and let $c_\varepsilon  \geq 1$,
 $T_{\varepsilon}\geq 1$ such that  $
 c^{\frac{84}{103}}T = \varepsilon^{\frac{52\alpha}{103} -\frac{11}{103}}$ for some $\alpha$ with  $0<\alpha < \frac{11}{52}$.
Then the distribution of the tagged particle $\hat{f}_t^N$ at time $0\le t \le T_{\varepsilon}$ converges to the time dependent density $f_t$ in the space $L^1( \mathds{T}^3 \times \mathds{R}^3 )$, in the Boltzmann-Grad limit $N \varepsilon^2 = c$, where $f_t$ satisfies the linear Boltzmann equation $\eqref{1.1}$ for time $T_{\varepsilon}$ which diverging with $N \to \infty$. There exists $\varepsilon_0>0$  such that for every $\varepsilon >0$, with $ \varepsilon < \varepsilon _0 $ such that for any $t\in [0,T_{\varepsilon}]$
the error can be estimated by
\begin{align*}
\| \hat{f}^N_t(x,v) - f_t (x,v) \|_{L^1(\mathds{T}^3 \times \mathds{R}^3)} & \le
C \varepsilon ^{\alpha}.
\end{align*}
\end{theorem}
The following theorem is about the derivation of the linear heat equation from a Rayleigh gas particle system distributed according to a Maxwellian background.

\begin{theorem}\label{thm2}
Let $f_0 \in L^1(\mathds{T}^3 \times \mathds{R}^3)$ the initial distribution of the tagged particle and consider the background particles are distributed according to a Maxwellian $M_{\beta}$. Let $\rho(\tau,x)$ be the solution of the linear heat equation
\begin{align*}
 \partial_{\tau} \rho - \kappa_{\beta} \Delta _x \rho &=0 \quad \quad \mathrm{in}\ (0,\infty) \times \mathds{T}^3,\\
 \rho (0, \cdot)&= \rho_0 \quad \ \ \mathrm{in}\ \mathds{T}^3,
\end{align*}
where the diffusion parameter $\kappa_{\beta}$ is given by
\begin{equation*}
\kappa_{\beta} := \frac{1}{3} \int_{\mathds{R}^3}
v \mathcal{L}^{-1}v M_{\beta}(v)
\ \d v,
\end{equation*}
where $\mathcal{L}$ is the Linear Boltzmann operator and $\mathcal{L}^{-1}$ is its pseudo-inverse defined on $(\mathrm{Ker} \mathcal{L})^{ \perp}$. Then the distribution of the tagged particle $\hat{f}^N (c t,  x,v)$  converges in $L^1$-norm to $\rho(\tau,x) M_{\beta}(v)$, i.e.
\begin{align*}
 \| \hat{f}^N (c t,  x,v) - \rho(\tau,x) M_{\beta}(v) \|_{L^1 ([0,T] \times \mathds{T}^3 \times \mathds{R}^3)} \to 0,
\end{align*}
in the limit $N \to \infty$, with $c = N \varepsilon^2 \to \infty.$
\end{theorem}

\subsection*{Plan of paper}
In the next section we provide an overview of the main ingredients for the proofs of both theorems. This includes an introduction to collision trees and the evolution of probability measures on them as introduced in \cite{matt12}. In section 3 we prove Theorem \ref{thm:1.1.} by extending \cite{Matthies2018} with quantitative error estimates.
 The proof of Theorem \ref{thm2} is given in section 4. We adapt results from
 \cite{MR3455156}  and combine them with Theorem \ref{thm:1.1.}.

\section{Collision Trees, semigroups and related approaches}

The standard method to the representation of the dynamics of the particle system that makes a connection to the Boltzmann equation is the use of the BBGKY hierarchy. Substantial details about  BBGKY hierarchy method can be found in \cite{cerci88, cercignani94}.
 The $N$ particle distribution resulting from hard sphere dynamics at time $t$ is denoted by $f_N(t)$. Then, away from collisions, $f_N$ satisfies the Liouville equation,
\[ \partial_t f_N(t) + v \cdot \partial_x f_N(t)= 0. \]
By integrating and using a weak form, this can be represented,  away from collisions,
\begin{equation} \label{eq-bbgky}
 \partial_t f_N^{(s)}(t) + \sum_{1 \leq i \leq s } v_i \cdot \nabla _{x_i} f_N(t) = C_{s,s+1}f_N ^{(s+1)}(t),
\end{equation}
for $s=1,\dots,N$ where $f_N^{(s)}$ is the $s$ particle marginal and where $C_{s,s+1}$ is the effect on the distribution of the first $s$ particles by a collision with another particle given by
\[\mathcal C_{s,s+1} f^{(s+1)}(t,Z_s):= (N-s) \varepsilon^{d-1} \sum_{i=1}^s \int_{S^{d-1} \times \mathbb{R}^d} \nu \cdot (v_{s+1}-v_i)
 f_N^{(s+1)}(t,Z_s,x_i+ \varepsilon  \nu, v_{s+1}) \, \mathrm{d}\nu \, \mathrm{d}v_{s+1}.\]
 The system of $N$ equations \eqref{eq-bbgky} is known as the BBGKY hierarchy. If the initial distribution of the $N$ particles introduces correlations as,
$$f_N(0)=\frac{1}{\mathcal{Z}_N} 1_{\mbox{no}} f_0^{\otimes N},$$
where $1_{\mbox{no}}$ conditions on no initial overlap and $\mathcal{Z}_N$ is a normalising constant, then the initial distribution of $f_N^{(s)}$ is given by,
\[ f_N^{s}(0,Z_s) = \int f_N(0,Z_n) \, \mathrm{d}z_{s+1} \dots \, \mathrm{d}z_{N}.  \]
The system \eqref{eq-bbgky} is typically solved by some Duhamel principle and estimates often use graphical representations of iterated integral expressions of $\mathcal C_{s,s+1}$, see \cite{Gallagher2013} for simple trees and \cite{Bodineau2023a} for much richer structures.

Cluster expansions can deal with the physical trajectories for  hard-sphere flows in  \cite{Bodineau2022}. There, a cluster at time $t$ consists of all dynamical interactions. Two particles interact dynamically on $[0,t ]$ if they collide on that time interval.
Given a set of particle trajectories, a graph of dynamical interactions is built by adding an edge $\{i, j\}$ if two particles $i$ and $j$ collide. A cluster
of particle trajectories have a connected graph of dynamical interactions, and which do not interact dynamically with particles outside that set.

Clusters evolve forward in all particles and typically nearly all particles will be in the same cluster after a short time. In this paper we follow an approach developed in
\cite{matthies10, matt12} which is considering physical collision histories. It derives evolution equations and applies semigroup techniques to study the evolution of associated probabilities of  collision trees instead of the BBGKY hierarchy.

We summarise some key ingredients of that approach.
A collision tree or collision history $\Phi$ is a set that includes the collisions that the tagged particle experiences. More precisely, it includes the initial position and velocity of the tagged particle $(x_0, v_0) \in \mathds{T}^3 \times \mathds{R}^3$ along with a list of collisions that the tagged particle experiences. Each collision is denoted by $(t_j, \nu _j, v_j) \in (0,T] \times \mathbb{S}^2 \times \mathds{R}^3$, where $t_j$ is the time that the $j$-th collision happens, $\nu _j$ is the collision parameter and $v_j$ is the incoming velocity of the background particle.
\begin{definition}
The set of all collision trees $\mathcal{MT}$ is defined by
\begin{align*}
\mathcal{MT} :=\{ ((x_0, v_0), (t_1, \nu_1, v_1),...,(t_n,\nu_n, v_n)  ) : & \, (x_0,v_0)\in \mathcal{U} \times \mathds{R}^3,\\  &  \, t_i\in[0,T],\ \nu_i\in \mathbb{S}^2 ,\ v_i\in \mathds{R}^3, \ n\in \mathds{N}_0 \}.
\end{align*}
For a tree $\Phi \in \mathcal{MT}$ is defined the function $n(\Phi)$ to be the number of collisions in this tree and for $n\ge 1$ define $\bar{\Phi}$ as the collision history identical to $\Phi$ but with the final collision removed. Furthermore, we define the maximum collision time $\tau \in [0,T]$ as
\begin{equation*}
\tau = \tau(\Phi) :=
\begin{cases}
0, &n(\Phi)=0\\
\max\limits_{1\le j\le n} t_j, & \mathrm{else},
\end{cases}
\end{equation*}
we denote by $\bar{\Phi}$ the tree without the final collision in $\Phi$, and we define the marker of the final collision as,
$$(\tau, \bar{\nu }, \bar{v}):=(t_n, \nu_n , v_n).$$
\end{definition}

\definition For a collision history $\Phi \in \mathcal{MT}$, the maximum velocity $\mathcal{V}(\Phi) \in [0, \infty)$ in the history is defined as
$$\mathcal{V}(\Phi):= \max \big\{ \max_{j=1,...,n(\Phi)} |v_j|, \max_{s\in [0,T]} |v(s)|  \big\},$$
where $v_j$, for $j=1,...,n(\Phi)$ is the velocity of each background particle in the collision history $\Phi$ and $v(s),$ $s \in [0,T]$ is the velocity of the tagged particle in the same history for different times.

We associate probabilities of finding a given collision tree. We first define the idealised distribution.
Let $\Phi \in \mathcal{MT}$, then $P_0(\Phi)$ is zero unless $\Phi$ involves no collisions, in which case $P_0(\Phi)$ is given by initial distribution $f_0(u_0,v_0)$. $P_t(\Phi)$ remains zero until $t=\tau$ when there is an instantaneous increase to a positive value depending on $P_\tau(\bar{\Phi})$ and the final collision in $\Phi$. For $t>\tau$, $P_t(\Phi)$ decreases at a rate that is obtained by considering all possible collisions. This can be expressed as
\begin{equation} \label{eq-id}
\begin{cases}
\partial _t P_t (\Phi) & = c \mathcal{Q}_t[P_t](\Phi)
= c \mathcal{Q}_t^+ [P_t](\Phi) - c \mathcal{Q}^{-}_t[P_t](\Phi), \\ P_0(\Phi) & = f_0 (x_0,v_0)1_{n(\Phi)=0},
\end{cases}
\end{equation}
where
\begin{equation*}
\mathcal{Q}_t^+ [P_t](\Phi): = \left\{\begin{array}{cc}\delta(t-\tau)P_t (\bar{\Phi})g_0(v')[(v(\tau^-)-v')\cdot \nu]_+&\mbox{if } n(\Phi)\geq 1,
\\ 0&\mbox{if } n(\Phi)=0,\end{array}\right.
\end{equation*}
and
\begin{equation*}
\mathcal{Q}^{-}_t[P_t](\Phi): = P_t(\Phi) \int_{\mathbb{S}^{2}} \int_{\mathbb{R}^3} g_0(\bar{v})[(v(\tau)-\bar{v})\cdot \nu ]_+ \, \mathrm{d}\bar{v} \, \mathrm{d}\nu.
\end{equation*}
The evolution equation \eqref{eq-id} is well-posed by \cite[Thm 3.1]{Matthies2018}. Before we can introduce the empirical distribution related to the particle model, we need to introduce some further notation.
\definition A history $\Phi \in \mathcal{MT}$ is called non-grazing if
\begin{equation*}
\min_{1\le j \le n(\Phi)} \nu_j \cdot ( v(t_j^-) - v_j)>0.
\end{equation*}
This means that all the collisions in the history $\Phi$ are non-grazing, i.e., the tagged particle and each background particle $j$ are not flying parallel for long time.
\definition We say that a collision history $\Phi \in \mathcal{MT}$ is free from initial overlap at diameter $\varepsilon$ if initially the tagged particle is at least $\varepsilon$ away from the centre of each background particle. That is to say, for all $j=1,...,n(\Phi),$ $$|x_0-x_j|>\varepsilon.$$
Define $S(\varepsilon) \subset \mathcal{MT}$ to be the set of all histories that are free from initial overlap at radius $\varepsilon$.

\definition A collision history $\Phi \in \mathcal{MT}$ is called re-collision free at diameter $\varepsilon$ if for all $j=1,...,n(\Phi)$ and for all $t\in [0,T] \setminus \{t_j\}$, $$ |x(t) - (x_j +tv_j)|>\varepsilon.$$
That is, if the tagged particle and a background particle $j$ collide at time $t_j$ then the tagged particle has not previously collided and will not re-collide with the background particle $j$ up to time $T$.
Define the set
\begin{equation*}
R(\varepsilon):= \{ \Phi \in \mathcal{MT} :\ \Phi \text{\ is\ re-collision\ free\ at\ diameter}\ \varepsilon \}.
\end{equation*}

\definition \label{1.3.6.} The set of good histories $\mathcal{G}(\varepsilon)$ of diameter $\varepsilon$ is defined by
$$\mathcal{G}(\varepsilon) := \{ \Phi \in \mathcal{MT} : \ n(\Phi) \le M(\varepsilon),\ \mathcal{V}(\Phi) <V(\varepsilon), \ \Phi \in R(\varepsilon) \cap S(\varepsilon) \ \mathrm{and} \ \Phi  \text{ is\ non-}\mathrm{grazing} \},$$
for any decreasing functions $V, M : (0,\infty) \to [0, \infty)$ such that $\lim_{\varepsilon \to 0}V(\varepsilon) = \lim_{\varepsilon \to 0}M(\varepsilon) = \infty.$

On the set of good histories we can express the particle dynamics in a similar fashion to \eqref{eq-id}.
We consider now the empirical distribution on collision histories $\hat{P_t^{\varepsilon}}$ defined by the dynamics of the particle system for particles with diameter $\varepsilon$. We will write $\hat{P_t}$ instead of $\hat{P_t^{\varepsilon}}$. The main result is that $\hat{P_t}$ solves the differential equation $\eqref{3}$ below which is similar to the idealised equation \eqref{eq-id}. 

Now we define the operator $\hat{\mathcal{Q}}_t$ which mirrors the idealised operator in the empirical case. For given history $\Phi$, a time $0<t<T$ and $\varepsilon >0$, define the function $\mathds{1}^{\varepsilon}_t[\Phi] : \mathds{T}^3 \times{\mathds{R}^3} \to \{ 0,1\}$ by
\begin{align}
\mathds{1}^{\varepsilon}_t[\Phi] (\bar{x}, \bar{v}) :=
\begin{cases}
1 & \text{if\ for\ all\ } s\in (0,t),\ |x(s)-(\bar{x} + s \bar{v})|>\varepsilon,\\
0 &\text{else}.
\end{cases}
\end{align}
That is to say, $ \mathds{1}^{\varepsilon}_t[\Phi](\bar{x}, \bar{v})$ is $1$ if a background particle starting at the position $(\bar{x}, \bar{v})$ avoids colliding with the tagged particle defined by the collision history $\Phi$ up to time $t$. For a history $\Phi$, $t\ge 0$ and $\varepsilon<0$ define the gain operator,
\begin{align*}
\hat{\mathcal{Q}}^+_t[\hat{P_t}](\Phi):=
\begin{cases}
\delta (t- \tau) \hat{P_t}(\bar{\Phi}) \frac{ g_0(v') [(v(\tau^-) - v') \cdot \nu]_+ }{ \int_{\mathds{T}^3 \times \mathds{R}^3}g_0 (\bar{v}) \mathds{1}^{\varepsilon}_t[\Phi] (\bar{x}, \bar{v}) \, \d\bar{x}  \, \d\bar{v}} & \text{if}\ n\ge 1\\
0 & \text{if}\ n=0.
\end{cases}
\end{align*}
Next define the loss operator,
\begin{align*}
\hat{\mathcal{Q}}^-_t[\hat{P_t}](\Phi):=
\hat{P_t}(\Phi) \frac{   \int_{\mathbb{S}^2} \int_{\mathds{R}^3}  g_0(\bar{v}) [ (v(\tau) - \bar{v}) \cdot \nu]_+ \, \d\bar{v} \, \d\nu -\hat{C} (\varepsilon) }{ \int_{\mathds{T}^3 \times \mathds{R}^3}g_0 (\bar{v}) \mathds{1}^{\varepsilon}_t[\Phi] (\bar{x}, \bar{v}) \, \d\bar{x}\, \d\bar{v}},
\end{align*}
for some $\hat{C}(\varepsilon) >0$ depending on $t$ and $\Phi$ of $ o(1)$ as $\varepsilon$ tends to zero detailed later.\\ Finally define the operator $\hat{\mathcal{Q}_t}$ as
\begin{equation*}
\hat{\mathcal{Q}_t} := \hat{\mathcal{Q}}_t^+ - \hat{\mathcal{Q}}_t^-.
\end{equation*}

\begin{theorem} \cite[Theorem 4.6.]{Matthies2018}  For $\varepsilon$ sufficiently small and for $\Phi \in \mathcal{G(\varepsilon)}$, $\hat{P_t}$ solves the following equation
\begin{align}\label{3}
\begin{cases}
\partial_t{\hat{P_t}}(\Phi) &= (c -\gamma(t)) \hat{\mathcal{Q}_t}[\hat{P_t}] (\Phi) \\
\hat{P_0}(\Phi) &= \zeta(\varepsilon) f_0 (x_0, v_0) \mathds{1}_{n(\Phi)=0}.
\end{cases}
\end{align}
The functions $\gamma$ and $\zeta$ are given by

\begin{equation}\label{2.19.}
\zeta(\varepsilon) := (1- \frac{4}{3} \pi \varepsilon^3)^N,
\end{equation}
and
\begin{align}
\gamma(t)):=
\begin{cases}
n(\bar{\Phi}) \varepsilon^2, & \text{if}\  t=\tau \\
n(\Phi) \varepsilon^2, & \text{if}\  t>\tau.
\end{cases}
\end{align}
\end{theorem}

The evolutions of the idealised distribution \eqref{eq-id} and the empirical distribution \eqref{3} allow now a detailed understanding and analysis of the errors.

\section{Quantitative errors for Rayleigh gas}

 To prove Theorem \ref{thm:1.1.} we
  extend the work of \cite{Matthies2018} by looking  at error estimates over longer times considering long-term dynamics. We are doing this by finding a quantitative error for the difference $ |P_t(S)- \hat{P_t}(S)|$, detailed expressions are  in Proposition $\ref{2.3.1.}$ below. We first provide quantitative estimates for the energy of the tagged particle, the expected number of collisions and  estimates for recollisions. It will be enough to provide all these estimates for idealised evolution.

 \subsection{Energy estimates}
 For this,
let us define the kinetic energy
\begin{align*}
M_f(t):=\int_{\mathds{T}^3} \int_{\mathds{R}^3}  f_t(x,v) (1+|v|^2) \, \mathrm{d}v \, \d x
\end{align*}
and also define the momentum
\begin{align*}
D_f(t) := \int_{\mathds{T}^3} \int_{\mathds{R}^3}  f_t(x,v)|v| \, \mathrm{d}v\, \d x.
\end{align*}
Then
\begin{align*}
M_f(t)=& \int_{\mathds{T}^3} \int_{\mathds{R}^3}  f_t(x,v)\, \mathrm{d}v \, \d x+
\int_{\mathds{T}^3} \int_{\mathds{R}^3}  f_t(x,v)|v|^2 \mathrm{d}v \, \d x
=  1+ \int_{\mathds{T}^3} \int_{\mathds{R}^3}  f_t(x,v)|v|^2 \mathrm{d}v \, \d x,
\end{align*}
since $f_t$ is a probability density. Now set
\begin{align*}
E(t) :=\int_{\mathds{T}^3} \int_{\mathds{R}^3}  f_t(x,v)|v|^2 \mathrm{d}v \, \d x.
\end{align*}
We introduce two constants $M_g$ and $C_g$. They are finite, as the probability density $g_0$ is admissible.
\begin{align}
\int_{\mathbb{S}^2} \int_{\mathds{R}^3}g_0(\bar{v}') |\bar{v}'|^2 \ \d\bar{v}' \, \d{\nu} =: M_g \quad \textrm{and} \quad
\int_{\mathbb{S}^2} \int_{\mathds{R}^3}g_0(\bar{v}') |\bar{v}'|^3 \ \d\bar{v}' \, \d{\nu} =: C_g. \label{g-const}
\end{align}
\begin{proposition}\label{1.1.1}
For all $t \ge 0$, it holds true that
\begin{align*}
M_f(t)  \le 1+ E_0^2 + (M_gE_0 +C_g)ct +\left(M_g+ \frac{C_g}{E_0}\right)^2 \frac{c^2 t^2}{2},
\end{align*} where $E_0 := \max \{ 1, E(0)\}$.
\end{proposition}
\begin{proof}
Taking the derivative of $E(t)$ we get
\begin{align*}
\frac{\d}{\d t}E(t) &= \frac{\d}{\d t}\int_{\mathds{T}^3} \int_{\mathds{R}^3}  f_t(x,v)|v|^2 \,  \mathrm{d}v \, \d x\\
&= \int_{\mathds{T}^3} \int_{\mathds{R}^3} -v\cdot \nabla_x f_t(x,v)|v|^2\, \mathrm{d}v \, \d x\\
&+ c \int_{\mathds{T}^3}\int_{\mathds{R}^3} |v|^2 \int_{\mathbb{S}^2} \int_{\mathds{R}^3}   f_t(x,v') g_0(\bar{v}') [(v-\bar{v})\cdot \nu]_+ \, \d \bar{v}   \, \d {\nu}\, \mathrm{d}v \, \d x\\
&- c \int_{\mathds{T}^3}\int_{\mathds{R}^3} |v|^2 f_t(x,v) \int_{\mathbb{S}^2} \int_{\mathds{R}^3}  g_0(\bar{v}) [(v-\bar{v})\cdot \nu]_+  \, \d \bar{v} \, \d {\nu}\, \mathrm{d} v \, \d x,
\end{align*}
where we used the linear Boltzmann equation $\eqref{1.1}$. Now we observe that it always holds true that $$|v|^2 \le |v'|^2 + |\bar{v}'|^2,$$
where $v$ is the post-collisional velocity and $ v'$ and $\bar{v}'$ are the pre-collisional velocities.
Thus,
\begin{align*}
\frac{\d}{\d t}E(t)
&\le  c \int_{\mathds{T}^3}\int_{\mathds{R}^3} (|v'|^2 + |\bar{v}'|^2) \int_{\mathbb{S}^2} \int_{\mathds{R}^3}   f_t(x,v') g_0(\bar{v}') [(v-\bar{v})\cdot \nu]_+ \, \d \bar{v}   \, \d {\nu}\, \mathrm{d}v \, \d x\\
&- c \int_{\mathds{T}^3}\int_{\mathds{R}^3} |v|^2 f_t(x,v) \int_{\mathbb{S}^2} \int_{\mathds{R}^3}  g_0(\bar{v}) [(v-\bar{v})\cdot \nu]_+ \, \d \bar{v} \, \d{\nu} \, \mathrm{d}v \, \d x\\
& = c \int_{\mathds{T}^3}\int_{\mathds{R}^3} 
   (Q^+[f|v'|^2]- Q^-[f|v|^2])\, \mathrm{d}v \, \d x\\
&+ c \int_{\mathds{T}^3}\int_{\mathds{R}^3} \int_{\mathbb{S}^2} \int_{\mathds{R}^3}f_t(x,v') g_0(\bar{v}') |\bar{v}'|^2 [(v-\bar{v})\cdot \nu]_+ \ \d \bar{v}  \, \d {\nu}\, \mathrm{d}v \, \d x.
\end{align*}
In the last equality, the first term is equal to zero, by the conservation of mass. Now, by using change of coordinates $\bar{v}, v \to \bar{v}',v'$ we get that the last integral becomes
\begin{align*}
&\int_{\mathds{T}^3}\int_{\mathds{R}^3} \int_{\mathbb{S}^2} \int_{\mathds{R}^3}f_t(x,v') g_0(\bar{v}') |\bar{v}'|^2 [(v-\bar{v})\cdot \nu]_+ \, \d \bar{v} \, \d {\nu}\, \mathrm{d}v \, \d x\\
&= \int_{\mathds{T}^3}\int_{\mathds{R}^3} \int_{\mathbb{S}^2} \int_{\mathds{R}^3}f_t(x,v') g_0(\bar{v}') |\bar{v}'|^2 [(v'-\bar{v}')\cdot \nu]_+ \, \d \bar{v}' \, \d {\nu}\, \mathrm{d}v' \, \d x.
\end{align*}
Therefore,
\begin{align*}
\frac{\d}{\d t}E(t) &\le  c \int_{\mathds{T}^3}\int_{\mathds{R}^3} \int_{\mathbb{S}^2} \int_{\mathds{R}^3}f_t(x,v') g_0(\bar{v}') |\bar{v}'|^2 [(v'-\bar{v}')\cdot \nu]_+ \, \d \bar{v}' \, \d {\nu}\, \mathrm{d}v' \, \d x\\
& \le c \int_{\mathds{T}^3}\int_{\mathds{R}^3} \int_{\mathbb{S}^2} \int_{\mathds{R}^3}f_t(x,v') g_0(\bar{v}') |\bar{v}'|^2 (|v'|+|\bar{v}'|) \, \d \bar{v}' \, \d {\nu}\ \mathrm{d}v' \, \d x\\
&= c \int_{\mathds{T}^3}\int_{\mathds{R}^3}  f_t(x,v') |v'| \, \mathrm{d}v' \, \d x\int_{\mathbb{S}^2} \int_{\mathds{R}^3}g_0(\bar{v}') |\bar{v}'|^2 \, \d \bar{v}'   \, \d {\nu} \\
&\quad+ c \int_{\mathds{T}^3}\int_{\mathds{R}^3}  f_t(x,v')\, \mathrm{d}v' \, \d x \int_{\mathbb{S}^2} \int_{\mathds{R}^3}g_0(\bar{v}') |\bar{v}'|^3 \, \d \bar{v}' \, \d {\nu}.
\end{align*}
Therefore using the constants in \eqref{g-const},
\begin{align*}
\frac{\d}{\d t}E(t) &\le c M_g \int_{\mathds{T}^3}\int_{\mathds{R}^3}  f_t(x,v') |v'| \, \mathrm{d}v' \, \d x+C_g
= c M_g \int_{\mathds{T}^3}\int_{\mathds{R}^3}  f_t(x,v')^{\frac{1}{2}} f_t(x,v')^{\frac{1}{2}} |v'| \, \mathrm{d}v' \, \d x+cC_g\\
&\le c M_g \left(\int_{\mathds{T}^3}\int_{\mathds{R}^3}  f_t(x,v') |v'|^2 \, \mathrm{d}v' \, \d x\right)^{\frac{1}{2}} \left(\int_{\mathds{T}^3}\int_{\mathds{R}^3}  f_t(x,v')\, \mathrm{d}v' \, \d x\right)^{\frac{1}{2}} + c C_g\\
&= c M_g E(t)^{\frac{1}{2}} + c C_g,
\end{align*}
by the Cauchy-Schwarz inequality and the fact that $f_t$ is a probability density. Therefore we get
\begin{align*}
\frac{\d}{\d t}E(t) &\le c M_g \sqrt{ E(t)} + c C_g.
\end{align*}
This immediately leads  to
\begin{align*}
\int_{\mathds{T}^3} \int_{\mathds{R}^3}  f_t(x,v)|v|^2 \, \mathrm{d}v \, \d x& \le \left(E_0 + \big(\frac{M_g}{2}+ \frac{C_g}{2E_0}\big) c t \right)^2\\
& = E_0^2 + (M_gE_0 +C_g) c t +\left(M_g+ \frac{C_g}{E_0}\right)^2 \frac{c^2t^2}{2}, \quad \forall \ t\ge 0,
\end{align*}
where by $E_0 := \max \{ 1, E(0)\}$.
Thus,
\begin{align*}
M_f(t) &= 1+ \int_{\mathds{T}^3} \int_{\mathds{R}^3}  f_t(x,v)|v|^2 \, \mathrm{d}v \, \d x\\
&  \le 1+ E_0^2 + (M_gE_0 +C_g)c t +\left(M_g+ \frac{C_g}{E_0}\right)^2 \frac{c^2 t^2}{2}, \quad \forall \ t\ge 0.
\end{align*}
\end{proof}

\begin{lemma} \label{1.1.2.} For all $t\ge 0$, it holds true that
\begin{align*}
D_f(t) \le \sqrt{E(t)}.
\end{align*}
\end{lemma}
\begin{proof}
By taking the definition of $D_f(t)$ and using the Cauchy-Schwarz inequality we take
\begin{align*}
D_f(t) = \int_{\mathds{T}^3} \int_{\mathds{R}^3}  f_t(x,v)|v| \, \mathrm{d}v \, \d x
&= \int_{\mathds{T}^3} \int_{\mathds{R}^3} |v|  f^{\frac{1}{2}}_t(x,v) f^{\frac{1}{2}}_t(x,v)\, \mathrm{d}v \, \d x\\
&\le \left( \underbrace{\int_{\mathds{T}^3} \int_{\mathds{R}^3} |v|^2  f_t(x,v)  \, \d  v \,  \d x}_{= E(t)} \right)^{\frac{1}{2}} \underbrace{\left( \int_{\mathds{T}^3} \int_{\mathds{R}^3} f_t(x,v)\,  \d v\,  \d x\right)^{\frac{1}{2}} }_{=1}\\
& = \sqrt{E(t)}.
\end{align*}
Therefore
\begin{align*}
D_f(t) & \le \sqrt{E(t)} \le \sqrt{ E_0^2 + (M_gE_0 +C_g)ct +\left(M_g+ \frac{C_g}{E_0}\right)^2 \frac{c^2t^2}{2}}, \quad \forall \ t\ge 0.
\end{align*}
\end{proof}
So Lemma \ref{1.1.2.}  implies that the momentum of the particles grows at most linearly for all times.
\subsection{Number of collisions}
Now, we want to find a bound for $\mathbb{E}(n(\Phi))(t)$, where $n(\Phi)$ is the number of collisions in the collision history $\Phi$.~We will try to find an evolution equation for $\mathbb{E}(n(\Phi))(t)$. For that we will use \eqref{eq-id} for the evolution of the idealised distribution. 
The idealised equation should be thought as equivalent of the linear Boltzmann equation but written on the space $\mathcal{MT}$ instead of on the phase-space $\mathds{T}^3 \times \mathds{R}^3$.

\begin{lemma} The expected number of collisions satisfies the estimate
\begin{align} \label{2.2}
\mathbb{E}(n(\Phi))(t)
 \le
 1+  \pi \left( (\beta + E_0) ct
 + \frac{1}{2 \sqrt{2}} (M_g + \frac{C_g}{E_0}) c^2 t^2 \right),
  \quad \forall t\ge 0.
\end{align}
\end{lemma}
\begin{proof}
 We observe that at the initial time $t=0$ there is only the tagged particle in the tree, since there is no collision yet, i.e., no initial overlap in the idealised evolution. Thus,
$\mathbb{E}(n(\Phi))(0)=1$, and  for general $t$ we set
 $$\mathbb{E}(n(\Phi))(t) = \int_{\mathcal{MT}} n(\Phi) P_t(\Phi) \, \d \Phi.$$
Hence,
\begin{align*}
\frac{\d}{\d t}\mathbb{E}(n(\Phi))(t) = \frac{\d}{\d t}\int_{\mathcal{MT}} n(\Phi) P_t(\Phi)\, \d \Phi
& =\int_{\mathcal{MT}}  \frac{\d}{\d t} n(\Phi) P_t(\Phi)\, \d \Phi+ \int_{\mathcal{MT}} n(\Phi) \frac{\d}{\d t}P_t(\Phi)\, \d \Phi\\
&= c \int_{\mathcal{MT}} n(\Phi)[ Q^+[P_t](\Phi) - P_t(\Phi)Q^-_{\tau}(\Phi) ]\, \d \Phi,
\end{align*}
here we used equation $\eqref{eq-id}$ and the fact that $\frac{\d}{\d t} n(\Phi)=0.$
Now,
\begin{align*}
\int_{\mathcal{MT}} n(\Phi) Q^+[P_t](\Phi)\, \d \Phi & = \int_{\mathcal{MT}} n(\Phi) \mathds{1}_{t=\tau(\Phi)} P_t(\bar{\Phi})g(\bar{v}) |v^{\varepsilon} (\tau^-) - \bar{v}|\, \d \Phi \\
&= \int_{\mathcal{MT}} n(\bar{\Phi}) \mathds{1}_{t=\tau(\Phi)} P_t(\bar{\Phi})g(\bar{v}) |v^{\varepsilon} (\tau^-) - \bar{v}|\, \d \Phi\\
&\quad+ \int_{\mathcal{MT}}\mathds{1}_{t=\tau(\Phi)} P_t(\bar{\Phi})g(\bar{v}) |v^{\varepsilon} (\tau^-) - \bar{v}|\, \d \Phi
\end{align*}
and
\begin{align*}
\int_{\mathcal{MT}} n(\Phi)P_t(\Phi)Q^-_{\tau}(\Phi) \, \d \Phi = \int_{\mathcal{MT}} n(\Phi)P_t(\Phi) \int_{\mathds{R}^3} \int_{B(0,1)} g(v_*)|v^{\varepsilon}(t) - v_* |\, \d {S}\, \d v_*\, \d \Phi.
\end{align*}
Thus,
\begin{align*}
\frac{\d}{\d t}\mathbb{E}(n(\Phi))(t) &=  c \int_{\mathcal{MT}} n(\Phi) \mathds{1}_{t=\tau(\Phi)} P_t(\bar{\Phi})g(\bar{v}) |v^{\varepsilon} (\tau^-) - \bar{v}|\, \d \Phi \\
& \quad- c \int_{\mathcal{MT}} n(\Phi)P_t(\Phi) \int_{\mathds{R}^3} \int_{B(0,1)} g(v_*)|v^{\varepsilon}(t) - v_* |\, \d {S}\, \d v_*\, \d \Phi\\
&\quad+  c \int_{\mathcal{MT}}\mathds{1}_{t=\tau(\Phi)} P_t(\bar{\Phi})g(\bar{v}) |v^{\varepsilon} (\tau^-) - \bar{v}| \, \d \Phi.
\end{align*}
We define as $h_t(\Phi) := n(\Phi) P_t(\Phi)$ and then we get
\begin{align*}
\frac{\d}{\d t}\mathbb{E}(n(\Phi))(t) &=  c \int_{\mathcal{MT}} [ Q^+ [h_t](\Phi) - h_t(\Phi)Q^-_{\tau}(\Phi)]\, \d \Phi
+  c \int_{\mathcal{MT}}\mathds{1}_{t=\tau(\Phi)} P_t(\bar{\Phi})g(\bar{v}) |v^{\varepsilon} (\tau^-) - \bar{v}|\, \d \Phi.
\end{align*}
Thus,
\begin{align*}
\frac{\d}{\d t}h_t(\Phi) &= c Q^+ [h_t](\Phi) - c h_t(\Phi)Q^-_{\tau}(\Phi)+ c \mathds{1}_{t=\tau(\Phi)} P_t(\bar{\Phi})g(\bar{v}) |v^{\varepsilon} (\tau^-) - \bar{v}|.
\end{align*}
We define the forcing term $\Delta(t, \varepsilon)=\Delta^{\varepsilon}(t):=\mathds{1}_{t=\tau(\Phi)} P_t(\bar{\Phi})g(\bar{v}) |v^{\varepsilon} (\tau^-) - \bar{v}|.$
Then the above equation becomes
\begin{align*}
\frac{\d}{\d t}h_t(\Phi) &= c Q^+ [h_t](\Phi) - c h_t(\Phi)Q^-_{\tau}(\Phi)+ c \Delta^{\varepsilon}(t).
\end{align*}
By using the variation of constant formula, we have
\begin{equation*}
h_t(\Phi) = c P_t(h_0(\Phi)) + c \int_{0}^{t} P_{t-s} ( \Delta^{\varepsilon} (s))\, \d s,
\end{equation*}
where $P_t$ is the solution semigroup of equation $\eqref{3}$. Integrating over $\mathcal{MT}$ we obtain
\begin{equation}\label{1.4}
\int_{\mathcal{MT}} h_t(\Phi) \, \d \Phi = c \int_{\mathcal{MT}} P_t (h_0(\Phi)) \, \d \Phi +
c\int_{\mathcal{MT}} \int_{0}^{t}P_{t-s} (\Delta ^{\varepsilon}(s)) \, \d s \,  \mathrm{d} \Phi,
\end{equation}
with
\begin{align*}
\int_{\mathcal{MT}} \int_{0}^{t}P_{t-s} (\Delta ^{\varepsilon}(s)) \, \d s \, \mathrm{d} \Phi  & = \int_{\mathcal{MT}} \left| \int_{0}^{t}P_{t-s} (\Delta ^{\varepsilon}(s)) \, \d s \right| \, \mathrm{d} \Phi
 =  \int_{0}^{t} \left| \int_{\mathcal{MT}} P_{t-s} (\Delta ^{\varepsilon}(s)) \, \d \Phi \right| \, \d s \\
& =  \int_{0}^{t} \big \| P_{t-s} (\Delta ^{\varepsilon}(s)) \big\|_{L^1(\mathcal{MT})} \, \d s 
= \int_{0}^{t} \big \| \Delta ^{\varepsilon}(s) \big\|_{L^1(\mathcal{MT})} \, \d s.
\end{align*}
Here, in the second equality we used Fubini's Theorem and in the last equality that the solution semigroup $P_{t-s}$ preserves the measure.
Now we will bound the term $\| \Delta ^{\varepsilon}(s) \|_{L^1(\mathcal{MT})}$.
\begin{align*}
\| \Delta ^{\varepsilon}(s)\|_{L^1(\mathcal{MT})} =& \int_{\mathcal{MT}} \mathds{1}_{s=\tau(\Phi)} P_s(\bar{\Phi})g(\bar{v}) |v^{\varepsilon} (\tau^-) - \bar{v}|\, \mathrm{d} \Phi\\
=&  \int_{\mathcal{MT}} \int_{\mathds{R}^3} \int_{B(0,1)} P_s(\bar{\Phi})g(\bar{v}) |v^{\varepsilon} (\tau^-) - \bar{v}| \, \d S \, \d \bar{v}\, \mathrm{d} \bar{\Phi} \\
 \le& \int_{\mathcal{MT}} \int_{\mathds{R}^3} \int_{B(0,1)}  P_s(\bar{\Phi})g(\bar{v}) (|v^{\varepsilon} (\tau^-)|+ |\bar{v}|) \, \d S \, \d \bar{v}\, \mathrm{d}\bar{ \Phi} \\
 =& \pi \int_{\mathcal{MT}} \int_{\mathds{R}^3} P_s(\bar{\Phi})g(\bar{v})|v^{\varepsilon} (\tau^-)|\, \d \bar{v} \, \mathrm{d} \bar{\Phi} 
 +\pi  \int_{\mathcal{MT}} \int_{\mathds{R}^3} P_s(\bar{\Phi})g(\bar{v}) |\bar{v}| \, \d \bar{v}\, \mathrm{d}\bar{ \Phi}.
\end{align*}
First we calculate the second term
\begin{align*}
\pi  \int_{\mathcal{MT}} \int_{\mathds{R}^3} P_s(\bar{\Phi})g(\bar{v}) |\bar{v}| \, \d \bar{v}\, \mathrm{d} \bar{ \Phi}
=
\pi   \int_{\mathds{R}^3}g(\bar{v}) |\bar{v}| \, \d\bar{v} \int_{\mathcal{MT}} P_s(\bar{\Phi}) \, \mathrm{d} \bar{\Phi}
=& \pi   \int_{\mathds{R}^3}g(\bar{v}) |\bar{v}| \, \d \bar{v}=\pi \beta.
\end{align*}
Here we used the fact that $P_s$ is a probability density on $ \mathcal{MT}$ and we defined $\beta : = \int_{\mathds{R}^3}g(\bar{v}) |\bar{v}| \, \d \bar{v} $.
Now we calculate the first term of $\| \Delta ^{\varepsilon}(s)\|_{L^1(\mathcal{MT})} $.
\begin{align*}
\pi \int_{\mathcal{MT}} \int_{\mathds{R}^3} P_s(\bar{\Phi})g(\bar{v})|v^{\varepsilon} (\tau^-)|\, \d \bar{v}\, \mathrm{d} \bar{\Phi}
=&\pi \int_{\mathcal{MT}} P_s(\bar{\Phi})|v^{\varepsilon} (\tau^-)| \mathrm{d} \bar{\Phi}  \int_{\mathds{R}^3} g(\bar{v}) \, \d \bar{v}\\
=&  \pi \int_{\mathcal{MT}} P_s(\bar{\Phi})|v^{\varepsilon} (\tau^-)| \, \mathrm{d} \bar{\Phi}
= \pi \int_{\mathds{T}^3} \int_{\mathds{R}^3}  f_s(x,v)|v| \, \mathrm{d}v \, \d x\\
\le & \pi \int_{\mathds{T}^3} \int_{\mathds{R}^3}  f_s(x,v) (1+|v|) \, \mathrm{d}v \, \d x
\end{align*}
Here we used the fact that $g$ is a probability density on $ \mathds{R}^3$ and \cite[Theorem 3.1]{Matthies2018}.
Now, by using the Lemma $\ref{1.1.2.}$ above we get
\begin{align*}
\| \Delta ^{\varepsilon}(s)\|_{L^1(\mathcal{MT})}
&\le \pi (\beta + D_f(s))\\
& \le \pi \left(\beta + \sqrt{E_0^2 +(M_g E_0+C_g)c s + \left(M_g +\frac{C_g}{E_0}\right)^2 \frac{c^2 s^2}{2}}\right).
\end{align*}
Thus, the equation $\eqref{1.4}$ becomes
\begin{align*}
\int_{\mathcal{MT}} h_t(\Phi) \, \d \Phi
& \le 1 + c \pi \int_{0}^{t} \left( \beta + \sqrt{E_0^2 +(M_g E_0+C_g) c s + \left(M_g +\frac{C_g}{E_0}\right)^2 \frac{c^2 s^2}{2}} \right)\, \d s\\
& \le 1 + c \pi \int_{0}^{t} \left( \beta + \sqrt{\left( E_0 + \left(M_g +\frac{C_g}{E_0}\right)\frac{c s}{\sqrt{2}} \right)^2 }\right)\, \d s\\
& \le 1+ c  \pi \left( (\beta + E_0)t 
+ \frac{1}{2 \sqrt{2}} (M_g + \frac{C_g}{E_0}) ct^2 \right).
\end{align*}
Therefore,
\begin{align*}
\mathbb{E}(n(\Phi))(t)
= \int_{\mathcal{MT}} h_t(\Phi) \, \d \Phi
 \le   1+  \pi \left( (\beta + E_0) ct
 + \frac{1}{2 \sqrt{2}} (M_g + \frac{C_g}{E_0}) c^2 t^2 \right).
\end{align*}
\end{proof}

\begin{proposition} \label{2.1.2.}
 Let $\eta >0$ sufficiently small, $\delta>0$ and let the decreasing functions $V, M : (0,\infty) \to [0, \infty)$ such that $\lim_{\varepsilon \to 0}V(\varepsilon) = \lim_{\varepsilon \to 0}M(\varepsilon) = \infty,$ then
\begin{align} \nonumber
& \mathds{P}_t \big(\text{ re-collisions with} \  n(\Phi)  \le M(\varepsilon)   \ and\ \mathcal{V}(\Phi) <V(\varepsilon) \big)\\
 &\le
 C M(\varepsilon) \left[ 1+ \frac{\eta V(\varepsilon)}{\delta^2} + \delta \right]\varepsilon (T V(\varepsilon))^2 + C M(\varepsilon) TV(\varepsilon) \left(\frac{1}{\eta } \right) \left( \frac{\varepsilon}{\eta}\right)^2 .\label{eqn:recol}
 \end{align}
\end{proposition}

\begin{proof}
This is the most involved of the estimates. Note that we are estimating events for the idealised distribution, so we consider idealised trajectories that have point particles that come close within a distance $\varepsilon$.
The main strategy here is to estimate the volume of possible velocities and use that the root marginal is given by the solution of the linear Boltzmann equation.

To this end, we will first estimate this volume of bad trees in the case of one collision and, secondly, in the case of $j\ge 2$ collisions. First, note that $\mathcal{MT}_0\setminus R(\varepsilon)= \emptyset$, where $\mathcal{MT}_0 :=\{ \Phi \in \mathcal{MT}  :\ n(\Phi) = 0\}$. Next consider  $\Phi \in \mathcal{MT}_1 := \{ \Phi \in \mathcal{MT} :\ n(\Phi)=1\}$, $\Phi = ((x_0, v_0), (\tau, \bar{\nu}, \bar{v}))$, where $(x_0,v_0)\in  \mathds{T}^3\times \mathds{R}^3 $, $\tau \in [0,T]$, $\bar{\nu}\in \mathbb{S}^2$, $\bar{v} \in \mathds{R}^3$. If $\Phi \in \mathcal{MT}_1\setminus R(\varepsilon)$, then due to the periodic boundary conditions  there exits $m \in \mathds{Z}^3 :$ $|m|\le 2TV(\varepsilon)$, $\varepsilon >0$ and $V \colon (0,1) \to \mathds{R}^+$ and $\exists \ s \in (\tau , T]$ and $\exists \ \tilde{\nu}$ such that
\begin{align}\label{onecoll}
x(s) +m+\varepsilon \tilde{\nu} = x_1(s), \quad \varepsilon >0,
\end{align}
where $x(s) = x_0 +\tau v_0 +(s-\tau ) v(\tau )$,
with $v(\tau) = v_0 -\nu (v_0-\bar{v})\cdot \nu $
and $x_1(s)=x_0+\tau v_0   + (s-\tau ) \bar{v}(\tau)$.
Then, \eqref{onecoll} is equivalent to
\begin{align*}
 x_0 +\tau v_0 
 +(s-\tau ) v(\tau )+m +\varepsilon \tilde{\nu}
 &=
  x_0+\tau v_0
  + (s-\tau) \bar{v}(\tau )
 \end{align*}
 which is equivalent to
 \begin{align*}
  m= \varepsilon \tilde{\nu}
  + (s- \tau )(\bar{v}(\tau) - v(\tau)) 
\end{align*}
Multiplying each side by $\nu$ we take
\begin{align*}
 m \cdot {\nu} =-\varepsilon\tilde{\nu} \cdot {\nu}+ (s- \tau )(\bar{v}(\tau)  - v(\tau )) \cdot {\nu}
\end{align*}
where $(\bar{v}(\tau )  - v(\tau )) \cdot {\nu}=0$ by the definition of $v(\tau)$ above.
Hence,
\begin{align*}
 m \cdot {\nu} &= -\varepsilon \tilde{\nu} \cdot {\nu}.
\end{align*}
By taking the absolute value 
\begin{align*}
| m \cdot {\nu} |
&\le
\varepsilon |\tilde{\nu} \cdot {\nu}|
 \le
   \varepsilon |\tilde{\nu}|  |{\nu}|  =  \varepsilon,
\end{align*}
as  $ \nu, \tilde{\nu} \in \mathbb{S}^2$.
Thus $$|  m \cdot {\nu} | \le \varepsilon \Rightarrow  m \cdot {\nu}  \in (-\varepsilon, \varepsilon).$$
Therefore $$ \mathrm{Recollision\ set} \subset \bigcup_{|m|\le 2TV(\varepsilon)} \{ {\nu} :\ m \cdot {\nu}  \in (-\varepsilon, \varepsilon) \}.$$
The basic geometry strategy here is that the set of re-collision histories is inside a disc of radius $2TV(\varepsilon)$ and thickness $2 \varepsilon$. Therefore, this means that recollisions due to the periodic boundary conditions for the first particle are bounded by
\begin{equation}\label{recoll-perbc}
\mathrm{vol( Recollision\ set )}\le C \varepsilon (T V(\varepsilon))^2.
\end{equation}
Then taking into account that this can happen for any particle we obtain
\begin{equation}\label{2.4}
\mathrm{vol( Recollision\ set )}\le C  M(\varepsilon) \varepsilon (T  V(\varepsilon))^2
\end{equation}
and we will ignore this effect in the larger trees too.

Now, let $j\ge 2$ and $\Phi \in \mathcal{MT}_j$ with $\Phi= ((x_0, v_0), (t_1, \nu_1, v_1),...,(t_n, \nu_n, v_n))$. If $\Phi \in \mathcal{MT}_j \setminus R(\varepsilon)$, then either two of the collisions correspond to the same background particle or the tagged particle will re-collide with one of the background particles at some time $s\in (\tau , T ], \ \tau >t_j.$

For the first case we have that there exist numbers $l,k$ with $2\le l\le j$ and $1\le k< l$ such that the $l$th and the $k$th collision corresponds to the same background particle. This implies that
\begin{align*}
v_l = v_k + \nu _l (v(t_l^{-}) - v_k) \cdot \nu_l,
\end{align*}
where $v_l$ and $v_k$ are the incoming velocities of the background particles before the $l$th and the $k$th collision respectively, $\nu_l$ is the collision parameter and $v(t_l^{-}) $ is the velocity of the tagged particle before the $l$th collision. Thus $v_l$ is determined by $v_k$, $\nu_l$ and $ v(t_l^{-}) $, so $v_l$ can only be in a set of zero measure.

Now we will deal with the most substantial case, i.e. there exist $s\in (t_j, T],\ m\in \mathds{Z}^3,\ \tilde{\nu}$ and $k:\ 1\le k< j$ such that
\begin{align}\label{jcoll }
x(s) +m +\varepsilon \tilde{\nu} =x_k  (s),
\end{align}
where $ x(s)=x_0 +t_1v_0 +(t_2-t_1) v(t_1) +...+ (s-t_j)v(t_j)$ and $x_k(s)= x_k(t_k) +(s-t_k)v_k(t_k)$.
Then on the torus, \eqref{jcoll }
is equivalent to
\begin{align}\label{eqn:recollcond}
 x_0 +t_1v_0 +(t_2-t_1) v(t_1) &+...+ (s-t_j)v(t_j)+ m +\varepsilon \tilde{\nu}
 = x_k(t_k) +(s-t_k)v_k(t_k),
 \end{align}
 which is equivalent to
 \begin{align*}
Y+(s-t_j)v(t_j)+ m  +\varepsilon \tilde{\nu} = (s-t_k)v_k(t_k), 
\end{align*}
where $Y : = (t_k -t_{k-1})v(t_{k-1}) +... +(t_j -t_{j-1})v(t_{j-1}).$
Thus,
\begin{align*}
(s-t_j)v(t_j)= -\varepsilon \tilde{\nu}+sv_k(t_k)-Y_m,
\end{align*}
where $Y_m : = Y+m+t_kv_k(t_k).$
Hence the velocity of the root particle must satisfy
$$v(t_j)= \frac{-\varepsilon \tilde{\nu}}{s-t_j}+\frac{Y_m-sv(t_k)}{t_j-s}.$$
For fixed $s\in (t_j , T]$,
$v(t_j)$ as above and fixed $m$, then $v(t_j)$ is contained into the "cylinder" around the curve defined by $$ \frac{Y_m-sv(t_k)}{t_j-s},$$
and $Y_m= (t_k -t_{k-1})v(t_{k-1}) +... +(t_j -t_{j-1})v(t_{j-1}) +m +t_k   v_k( t_k    )$, for $1\le k\le j.$ Thus,
\begin{align*}
|Y_m| &\le (t_k -t_{k-1})|v(t_{k-1})| +... +(t_j -t_{j-1})|v(t_{j-1})| +|m| + t_k  |v_k( t_k   )|\\
& \le (t_k -t_{k-1})V(\varepsilon) +... +(t_j -t_{j-1})V(\varepsilon) +2TV(\varepsilon)+ t_k  V(\varepsilon)\\
&= (t_j -t_{k-1})V(\varepsilon)+2TV(\varepsilon)+t_k   V(\varepsilon) \le TV(\varepsilon)+2TV(\varepsilon)+TV(\varepsilon).
\end{align*}
Therefore,
\begin{equation*}
|Y_m|\le 4TV(\varepsilon), \ \varepsilon>0.
\end{equation*}

Let $\eta$ be 
sufficient small and split the interval $(t_j,T]$ into two parts. The first part is $(t_j,t_j+\eta)$ and the second part is $[t_j+\eta, T].$ 

First, consider $s\in [t_j+\eta, T]$ and define $r(s):= \frac{Y_m(t)-sv(t_k)}{t_j-s}$, which is a differentiable curve for $s>0$. Then,
$$\frac{\d}{\d s}r(s)= \frac{Y_m(t)-sv(t_k)}{(t_j-s)^2}- \frac{v(t_k)}{ t_j-s}.$$
Thus,
\begin{align*}
\left|\frac{\d}{\d s}r(s)\right|&\le  \frac{|Y_m(t)-sv(t_k)|}{(t_j-s)^2}+ \frac{|v(t_k)|}{| t_j-s|}
\le  \frac{4TV(\varepsilon)+TV(\varepsilon)}{(t_j-s)^2}+ \frac{V(\varepsilon)}{| t_j-s|}.
\end{align*}
Hence,
\begin{align*}
\left|\frac{\d}{\d s}r(s)\right| & \le  \frac{5TV(\varepsilon)}{(t_j-s)^2}+ \frac{V(\varepsilon)}{| t_j-s|}.
\end{align*}
Therefore the length of the curve is bounded by
\begin{align*}
\int_{t_j+\eta}^{T}\left|\frac{\d}{\d s}r(s)\right|\ \d s
& \le
\int_{t_j + \eta}^{T} \frac{5TV(\varepsilon)}{(t_j-s)^2}\ \d s + \int_{t_j + \eta}^{T}\frac{V(\varepsilon)}{| t_j-s|}\ \d s,
\end{align*}
where
\begin{align*}
 \int_{t_j + \eta}^{T} \frac{5TV(\varepsilon)}{(t_j-s)^2}\ \d s
 = (5TV(\varepsilon))\left[ \frac{1}{t_j -T} -\frac{1}{-\eta}\right]
&\le 5TV(\varepsilon)\left[\frac{1}{\eta  }\right].
\end{align*}
The second integral is
\begin{align*}
\int_{t_j + \eta}^{T}\frac{V(\varepsilon)}{| t_j-s|}\ \d s
= V(\varepsilon) [\log(T-t_j)-\log \eta]
& \le V(\varepsilon)[ \log T - \log \eta ]
= V(\varepsilon) \log \left(\frac{T}{\eta}\right) .
\end{align*}
Hence,
\begin{align*}
\mathrm{length} (r(s)) \le 5TV(\varepsilon)\left(\frac{1}{\eta }\right) + V(\varepsilon)\log \left(\frac{T}{\eta }\right)
& < 6TV(\varepsilon)\frac{1}{\eta } 
=: \hat{C}(\eta),
\end{align*}
where we just used 
$ \log (x) < x , \ \forall x>0. $
Therefore, there exists $C>0$ such that
\begin{equation}\label{2.6.}
\mathrm{vol}(v(t_j)) \le C \hat{C}(\eta) \left( \frac{\varepsilon}{\eta}\right)^2, \ s\in [t_j +\eta, T].
\end{equation}
Thus the volume of suitable root velocities
$$v(t_j)= \frac{\varepsilon (\nu-\tilde{\nu})}{s-t_j}+\frac{Y_m-sv_k(t_k)}{t_j-s}$$
 is bounded by the volume of the "cylinder" around the curve defined by $$ \frac{Y_m-sv_k(t_k)}{t_j-s},$$
where $Y_m= (t_k -t_{k-1})v(t_{k-1}) +... +(t_j -t_{j-1})v(t_{j-1}) +m +t_k   v_k( t_k )$, for $1\le k\le j$ and $ 2 \le j \le n(\Phi)$.

Secondly, consider the case where
$s \in (t_j, t_j  + \eta )$.  
In this interval, the recollisions can happen for most root velocities, so we will estimate the volume of the set of re-collisions mainly in terms of the velocity and the collision parameter for the recolliding particle $k$. We rewrite
\eqref{eqn:recollcond} by shifting the origin to $x(t_k)$ and with $m \in \mathds{Z}^3 $
 \begin{align}\label{eqn:jkrecol}
  (t_{k+1} -t_k) v(t_k) &+...+ (s-t_j)v(t_j)+\varepsilon \tilde{\nu}+m
 = (s-t_k)v_k(t_k).
\end{align}
We are expressing the explicit dependencies on the precollisional velocity  $v_k$ and parameter $\nu_k$ given the precollisional velocity of the tagged particle $v(t_k^-)$:
\begin{align*}
    v_k(t_k)&= v_k - \nu_k \cdot (v_k- v(t_k^-)) \nu_k,\\
    v(t_k)&= v(t_k^-) + \nu_k \cdot (v_k- v(t_k^-)) \nu_k.
 \end{align*}
As the later root velocities only depend on $v(t_k)$, this implies that the left hand side of \eqref{eqn:jkrecol} only depends on $\nu_k \cdot v_k$, while the right hand side only depends on the components of $v_k$ which are perpendicular to $\nu_k$. For fixed $\nu_k \cdot v_k$ and varying $s \in [t_j, t_j+\eta] $ and $\tilde{\nu} \in  \mathbb{S}^2$ the left hand side defines a cylinder of volume proportional to $\eta V(\varepsilon) \varepsilon^2$. We project this cylinder to the plane $\nu_k^\perp$ and its area is bounded by $C\eta V(\varepsilon) \varepsilon$ which yields the constraint for the orthogonal components of $v_k$. The factor  $s-t_k$ can be easily bounded from below by some $\delta>0$ with a probability $1-\delta$ as $k<j$. We estimate again the number of possible $m$ as above.
Then the overall volume of suitable $v_k$ and $\nu_k$ can be estimated
\begin{equation}\label{eqn:jkvolest}
\mathrm{vol}(v_k,\nu_k) \le C \frac{1}{\delta^2}\eta V(\varepsilon) \varepsilon (TV(\varepsilon))^2.
\end{equation}
Summing over all possible $k$ yields a factor $M(\varepsilon)$. 
We take now the probability of the set of re-collisions
\begin{align*}
 &\mathds{P}_t \big(\text{ re-collisions with}  \ n(\Phi) \le M(\varepsilon)  \ and\ \mathcal{V}(\Phi) <V(\varepsilon) \big)
 \\ &\le \sum_{j\ge 0} \mathds{P}_t \big( ( \mathcal{MT}_j \setminus \mathcal{R}(\varepsilon) \big)\\
 &= \mathds{P}_t \big( ( \mathcal{MT}_1 \setminus \mathcal{R}(\varepsilon) \big) + \sum_{j= 2}^{n(\Phi)} \mathds{P}_t \big( ( \mathcal{MT}_j \setminus \mathcal{R}(\varepsilon) \big)\\
 & \le C M(\varepsilon) \varepsilon (T V(\varepsilon))^2  + C \sum_{j= 2}^{n(\Phi)} M(\varepsilon) \left( TV(\varepsilon) \left(\frac{1}{\eta } \right) \left( \frac{\varepsilon}{\eta}\right)^2+  C \left[\frac{1}{\delta^2}\eta V(\varepsilon) +\delta \right]\varepsilon (TV(\varepsilon))^2 \right)\\
 & \le C M(\varepsilon) \left[ 1+ \frac{\eta M(\varepsilon) V(\varepsilon)}{\delta^2} + M(\varepsilon) \delta \right]\varepsilon (T V(\varepsilon))^2 + C (M(\varepsilon))^2 TV(\varepsilon) \left(\frac{1}{\eta } \right) \left( \frac{\varepsilon}{\eta}\right)^2,
 \end{align*}
as required, by using $\eqref{2.4}$, $\eqref{2.6.}$ and $\eqref{eqn:jkvolest}$.
\end{proof}
\subsection{Initial overlap}
We want to estimate the set of initial overlaps. For this, we have the following lemma.

\begin{lemma} \label{lem:overlap}
The probability of initial overlap of the hard spheres of diameter $\varepsilon$ is bounded by
 \begin{align*}
&\mathds{P}_t\left(\{ \Phi \in \mathcal{MT} : n(\Phi) \leq M(\varepsilon) \mbox{ and }  \exists j \in \{1,..., n(\Phi) \} : \ |x_0 - x_j| \le \varepsilon \}\right)\\&\le
C \varepsilon  + C M(\varepsilon)  \varepsilon^{3/2}.\numberthis \label{est:overlap}
\end{align*}
\end{lemma}
\begin{proof}
We consider a tree $\Phi \in \mathcal{MT}$ such that $\Phi = ((x_0,v_0),(t_1, \nu_1, v_1),..., (t_n, \nu_n, v_n ) )$, where $(x_0,v_0)\in  \mathds{T}^3\times \mathds{R}^3 $, $t_i \in [0,T]$, $\nu_i \in \mathbb{S}^2$, $v_i \in \mathds{R}^3$, for $i=1,...,n$. The condition for initial overlap of the hard spheres is
\[|x_j(0)- x_0 | \le \varepsilon. \]
Using $x_j(0)= x(t)- t_j v_j$ in the idealised setting, we obtain the condition $v_j \in B_j $ with
\[   B_j= \left\{ v \mid   | v - \frac{x(t)-x_0}{t_j} | \le  \frac{\varepsilon}{t_j} \right\}. \]
Assume $t_j \geq \delta$, then the probability of $v_j \in B_j$ is given by
\begin{align} \label{eqn:overlap1}
    \int_{B_j} g_0(v) \d v \leq C \left( \frac{\varepsilon}{\delta}\right)^3.
\end{align}
The probability to have a tree with two rapid collisions, i.e.~$t_2 <\delta$, can be bounded by estimating the rate of collisions for the root particle with velocity $v_0$
\[|Q^-_t(\Phi)|=\left|\int_{\mathds{R}^3} \int_{\mathbb{S}^2} g_0(\bar{v}) [(v_0-\bar{v})\cdot \omega ]_+  \d{\omega} \d \bar{v}\right| \leq C (1+ |v_0|), \]
due to the moment assumptions on $g_0$. Hence the probability to have two collisions is bounded by
\begin{align}\label{eqn:overlap2}
    C^2 \delta^2 (1+|v_0|).
\end{align}
The probability that the first colliding particle has initial overlap combines the probability of a single collision and overlap  which can be estimated by integrating over all initial velocities $|v_0|$
\begin{align}\label{eqn:overlap3}
\int_0^{T}  C \min\left(1, \frac{\varepsilon^3}{t^3} \right) \d t \le C \varepsilon.
 \end{align}
Setting $\delta= \sqrt{\varepsilon}$ and combining \eqref{eqn:overlap1}, \eqref{eqn:overlap2} and \eqref{eqn:overlap3} yields the result.
\end{proof}

\subsection{Error estimate}
\begin{proposition} \label{2.3.1.}  For any decreasing functions $V, M : (0,\infty) \to [0, \infty)$ such that $\lim_{\varepsilon \to 0}V(\varepsilon) = \lim_{\varepsilon \to 0}M(\varepsilon) = \infty$ and for any $t\in [0,T]$ and $\delta>0$ the probability of the set of bad trees has the following error estimate
\begin{align*}
\mathds{P}_t (\mathcal{MT} \setminus \mathcal{G}(\varepsilon))
& \le
\textnormal{Ov}(\varepsilon)
+
\textnormal{Rec}(\varepsilon)
+
\textnormal{Hi}(\varepsilon)
+
\textnormal{Vel}(\varepsilon),
\end{align*}
where
\begin{enumerate}
\item$\textnormal{Ov}(\varepsilon) \le
C \varepsilon   + C M(\varepsilon) \varepsilon^{3/2}$
is due to the initial overlaps,
\item$\textnormal{Rec}(\varepsilon) \le C M(\varepsilon) \left[ 1+ \frac{\eta M(\varepsilon) V(\varepsilon)}{\delta^2} + M(\varepsilon) \delta \right]\varepsilon (T V(\varepsilon))^2 + C (M(\varepsilon))^2 TV(\varepsilon) \left(\frac{1}{\eta } \right) \left( \frac{\varepsilon}{\eta}\right)^2
$
is due to re-collisions,
\item$ \textnormal{Hi}(\varepsilon) \le   \frac{1}{M(\varepsilon)} \left[
1+ c \pi \left( (\beta + E_0)t 
+ \frac{1}{2 \sqrt{2}} (M_g + \frac{C_g}{E_0}) c t^2 \right) \right] $ is due to the unbounded number of collisions and
\item$  \textnormal{Vel}(\varepsilon) \le
C_g  \frac{ M(\varepsilon) }{V(\varepsilon)^3} + \frac{M_{f_0}}{V(\varepsilon)^2}  + \frac{ M(\varepsilon)}{V(\varepsilon)^2}
\left(1+ E_0^2 + (M_gE_0 +C_g)ct +\left(M_g+ \frac{C_g}{E_0}\right)^2 \frac{c^2 t^2}{2}\right)
$ is due to the unbounded velocities,
 \end{enumerate}
 where $ C_g :=  \int_{\mathds{R}^3} g_0(v) |v|^3 \ \d v $ and $ M_{f_0} : = \int_{\mathds{T}^3 \times \mathds{R}^3 \setminus B(0,V(\varepsilon)) } f_0(x,v) |v|^2 \, \d x\, \d v,$ for $t\in [0,T]$.
\end{proposition}

\begin{proof}

By the inclusion-exclusion principle we obtain
\begin{align*}
\mathds{P}_t (\mathcal{MT} \setminus \mathcal{G}(\varepsilon))& \le \mathds{P}_t (\mathcal{MT}\setminus S(\varepsilon)) + \mathds{P}_t ( n(\Phi)>M(\varepsilon) )\\ & +  \mathds{P}_t \big( ( \mathcal{MT} \setminus \mathcal{G}(\varepsilon))\cap \{ \Phi :\ \Phi \in S(\varepsilon)\ \mathrm{with} \ n(\Phi) \le M(\varepsilon)  \ and\ \mathcal{V}(\Phi) <V(\varepsilon) \} \big)\\
& + \mathds{P}_t \big( \mathcal{V}(\Phi)>V(\varepsilon) \ \mathrm{with} \ n(\Phi) \le M(\varepsilon) \big).
\end{align*}
We observe that the third term can be written as
\begin{align*}
 \mathds{P}_t \big( ( \mathcal{MT} \setminus \mathcal{G}(\varepsilon))\cap \{ \Phi :\ & \Phi \in S(\varepsilon)\ \mathrm{with} \ n(\Phi) \le M(\varepsilon) \} \big)
 \\ &= \mathds{P}_t \big(\text{
 re-collisions \ with} \ n(\Phi) \le M(\varepsilon) \ and\ \mathcal{V}(\Phi) <V(\varepsilon)\big).\numberthis \label{1.6}
 \end{align*}
Now, by the Markov's inequality we have 
\begin{align*}
\mathds{P}_t \big( n(\Phi)  >M(\varepsilon) \big) &\le \frac{\mathbb{E}(n(\Phi))}{M(\varepsilon)}\\
& \le \frac{1}{M(\varepsilon)} \left[
1+ c \pi \left( (\beta + E_0)t 
+ \frac{1}{2 \sqrt{2}} (M_g + \frac{C_g}{E_0}) c t^2 \right)
 \right]. \numberthis \label{2.9}
\end{align*}
Now, we want to find a bound for the probability $\mathds{P}_t \big( \mathcal{V}(\Phi)>V(\varepsilon) \ \mathrm{with} \ n(\Phi) \le M(\varepsilon) \big)$, where $\mathcal{V}(\Phi)$ is the maximum velocity on the history $\Phi$ and $V$ is a decreasing function of $\varepsilon$ with $\lim_{\varepsilon \to 0}V(\varepsilon) = \infty.$ Let us consider for the background particles that the velocity $v_j$ is i.i.d. with respect to $g_0$ and assume furthermore that $\int_{\mathds{R}^3} g_0(v) |v|^3 \ \d v =: C_g.$ Then
\begin{align*}
\int_{\mathds{R}^3 \setminus B(0,V(\varepsilon))} g_0(v) \, \d v \le \frac{1}{V(\varepsilon)^3} \int_{\mathds{R}^3 \setminus B(0,V(\varepsilon))} g_0(v) |v|^3 \, \d v \le \frac{C_g}{V(\varepsilon)^3}.
\end{align*}
Thus
\begin{align*}
\mathds{P}_t \big( |v_j| \le  V(\varepsilon)  \big)= 1 - \mathds{P}_t \big( |v_j| >  V(\varepsilon)  \big) \ge 1 - \frac{C_g}{V(\varepsilon)^3}M(\varepsilon).
\end{align*}
Let us consider for the tagged particles that the initial velocity $v_0$ is i.i.d. with respect to $f_0$ and define furthermore  $$\int_{\mathds{T}^3 \times \mathds{R}^3 \setminus B(0,V(\varepsilon)) } f_0(x,v) |v|^2 \, \d x\, \d v =: M_{f_0} \ \textrm{and} \  \int_{\mathds{T}^3 \times \mathds{R}^3 \setminus B(0,V(\varepsilon))} f_t(x,v) |v|^2 \, \d x \, \d v =: \tilde{M}_f(t),$$ for $t\in [0,T]$. Then
\begin{align*}
\int_{\mathds{T}^3 \times \mathds{R}^3 \setminus B(0,V(\varepsilon))} f_0(x,v) \, \d x \, \d v \le \frac{1}{V(\varepsilon)^2} \int_{ \mathds{T}^3 \times \mathds{R}^3 \setminus B(0,V(\varepsilon))} f_0(v) |v|^2 \, \d x \, \d v = \frac{M_{f_0}}{V(\varepsilon)^2}
\end{align*}
and for $t \in [0,T]$
\begin{align*}
\int_{ \mathds{T}^3 \times \mathds{R}^3 \setminus B(0,V(\varepsilon))  } f_t(x,v) \, \d v \le \frac{1}{V(\varepsilon)^2} \int_{ \mathds{T}^3 \times \mathds{R}^3 \setminus B(0,V(\varepsilon)) } f_t(x,v) |v|^2 \,  \d x \, \d v = \frac{\tilde{M}_f(t)}{V(\varepsilon)^2}.
\end{align*}
Thus
\begin{align*}
\mathds{P}_t \big( & |v_0| >  V(\varepsilon) \ \mathrm{or} \ |v_1|> V(\varepsilon)\ \mathrm{or}\ ...\ \mathrm{or} \ |v_{n(\Phi)}| >  V(\varepsilon) \big)\\
&\le \mathds{P}_t \big( |v_0| >  V(\varepsilon) \big) +   \mathds{P}_t \big( |v_1|> V(\varepsilon) \big)
+ ... +  \mathds{P}_t \big( |v_{n(\Phi)}| >  V(\varepsilon) \big)\\
& \le \frac{M_{f_0}}{V(\varepsilon)^2} + n(\Phi) \frac{\tilde{M}_f(t)}{V(\varepsilon)^2}, \quad t \in [0,T].
\end{align*}
Hence
\begin{align*}
\mathds{P}_t \big( &\mathcal{V}(\Phi)>V(\varepsilon) \ \mathrm{with} \ n(\Phi)  \le M(\varepsilon) \big)\\ &\le \mathds{P}_t \big( |v_j|>V(\varepsilon) \big) +  \mathds{P}_t \big( |v(s)|>V(\varepsilon) \big)\\
&\le C_g  \frac{ M(\varepsilon) }{V(\varepsilon)^3} + \frac{M_{f_0}}{V(\varepsilon)^2} +  \tilde{M}_f(t) \frac{ M(\varepsilon)  }{V(\varepsilon)^2}\\
&\le
C_g  \frac{ M(\varepsilon) }{V(\varepsilon)^3} + \frac{M_{f_0}}{V(\varepsilon)^2} 
+ \frac{ M(\varepsilon)}{V(\varepsilon)^2}
\left(1+ E_0^2 + (M_gE_0 +C_g)ct +\left(M_g+ \frac{C_g}{E_0}\right)^2 \frac{c^2 t^2}{2}\right),
\numberthis \label{2.11.}
\end{align*}
by Proposition \ref{1.1.1}.
Therefore, by summing \eqref{eqn:recol}, \eqref{est:overlap}, \eqref{2.9},  and \eqref{2.11.} we obtain
 \begin{align*}
\mathds{P}_t (&\mathcal{MT} \setminus \mathcal{G}(\varepsilon))\\
& \le
C \varepsilon +  C M(\varepsilon)  \varepsilon^{3/2} \\&
\quad+  \frac{1}{M(\varepsilon)} \left[
1+ c \pi \left( (\beta + E_0)t 
+ \frac{1}{2 \sqrt{2}} (M_g + \frac{C_g}{E_0}) c t^2 \right)
 \right]\\
& \quad+
C M(\varepsilon) \left[ 1+ \frac{\eta M(\varepsilon)V(\varepsilon)}{\delta^2} + M(\varepsilon)\delta \right]\varepsilon (T V(\varepsilon))^2 + C (M(\varepsilon))^2 TV(\varepsilon) \left(\frac{1}{\eta } \right) \left( \frac{\varepsilon}{\eta}\right)^2
\\  &
\quad +
C_g  \frac{ M(\varepsilon) }{V(\varepsilon)^3} + \frac{M_{f_0}}{V(\varepsilon)^2}
 \quad + \frac{ M(\varepsilon)}{V(\varepsilon)^2}
\left(1+ E_0^2 + (M_gE_0 +C_g)ct +\left(M_g+ \frac{C_g}{E_0}\right)^2 \frac{c^2 t^2}{2}\right). \numberthis \label{Highterms}
\end{align*}
\end{proof}
\subsection{Quantitative error estimates}

Here, we introduce some extra notation. For $\varepsilon >0$, $\Phi \in \mathcal{G((\varepsilon)}$, $t\in [0,T]$, define
\begin{align*}
\eta^{\varepsilon}_t (\Phi) & := \int_{\mathds{T}^3 \times \mathds{R}^3} g_0 (\bar{v}) (1- \mathds{1}^{\varepsilon}_t[\Phi] (\bar{x}, \bar{v})) \,  \d \bar{x} \,  \d \bar{v},\\
R^{\varepsilon}_t (\Phi) &:= \zeta(\varepsilon) P_t(\Phi),\\
L(\Phi) & := - \int_{\mathbb{S}^2} \int_{\mathds{R}^3} g_0(\bar{v}) [(v(\tau) - \bar{v}) \cdot \nu]_+ \,  \d \bar{v} \,  \d \nu ,\\
C(\Phi) & := 2 \sup_{t\in [0,T]} \left\{  \int_{\mathbb{S}^2} \int_{\mathds{R}^3} g_0(\bar{v}) [(v(t) - \bar{v}) \cdot \nu]_+ \,  \d \bar{v} \,  \d \nu \right \} \\
\rho ^{\varepsilon , 0} _t (\Phi) & := \eta^{\varepsilon}_t (\Phi) C(\Phi) t.
\end{align*}
Further for $k \ge 1$ define,
\begin{align*}
\rho ^{\varepsilon , k} _t (\Phi) & := (1-\varepsilon) \rho ^{\varepsilon , k-1} _t (\Phi) + \rho ^{\varepsilon , 0} _t (\Phi) +\varepsilon,
\end{align*}
and, furthermore, define
\begin{align*}
\hat{\rho }^{\varepsilon } _t (\Phi) & := \rho ^{\varepsilon , n(\Phi)} _t (\Phi).
\end{align*}

Now, we want to find a bound for the difference $\hat{P_t} - P_t$. For this, we have the following lemma.

\begin{lemma} \label{2.4.3.} For any $\Phi \in \mathcal{G(\varepsilon)}$ and $\varepsilon$ sufficiently small,
\begin{align*}
\hat{\rho }^{\varepsilon } _t (\Phi)
&\le
C_1 C_2(1+M(\varepsilon)) T \varepsilon^2 (\beta + V(\varepsilon))^2 + M(\varepsilon) \varepsilon.
\end{align*}
\end{lemma}

\begin{proof}
Firstly, by \cite[Lemma 5.6]{Matthies2018} we obtain
\begin{align}\label{rho0}
\rho ^{\varepsilon , 0} _t (\Phi) & = \eta^{\varepsilon}_t (\Phi) C(\Phi) t \le C_1 C_2 T \varepsilon^2 (\beta + V(\varepsilon))^2.
\end{align}
Then for any $\Phi \in \mathcal{G(\varepsilon)}$,
\begin{align*}
\hat{\rho }^{\varepsilon } _t (\Phi)
&= \rho^{\varepsilon, n(\Phi)}_t = (1-\varepsilon)^{n(\Phi)} \rho^{\varepsilon, 0}_t(\Phi) + (\rho^{\varepsilon, 0}_t(\Phi) + \varepsilon) \sum_{ j=1}^{n(\Phi)} (1-\varepsilon)^{n(\Phi) -j} \\
& \le
\rho^{\varepsilon, 0}_t(\Phi) + (\rho^{\varepsilon, 0}_t(\Phi) + \varepsilon) \times n(\Phi)
\le \rho^{\varepsilon, 0}_t(\Phi) + (\rho^{\varepsilon, 0}_t(\Phi) + \varepsilon) M(\varepsilon)\\
&\le
C_1 C_2(1+M(\varepsilon)) T \varepsilon^2 (\beta + V(\varepsilon))^2 + M(\varepsilon) \varepsilon.
\end{align*}
as required by using \eqref{rho0}.
\end{proof}
\begin{theorem} \label{ThP} There are uniform constants $C$, $C_1$ and $C_2$ such that for every $t \in [0,T_{\varepsilon}]$ and $S \subset \mathcal{MT}$ measurable,
\begin{align*}
 &\sup_{S \subset \mathcal{MT}}|P_t(S)- \hat{P_t}(S)|\\
& \le  \frac{4}{3} \pi c \varepsilon
+C_1 C_2(1+M(\varepsilon)) T \varepsilon^2 (\beta + V(\varepsilon))^2 + M(\varepsilon) \varepsilon
+ C \varepsilon +  C M(\varepsilon)  \varepsilon^{3/2}\\&
\quad +  \frac{1}{M(\varepsilon)} \left[
1+ c \pi \left( (\beta + E_0)t 
+ \frac{1}{2 \sqrt{2}} (M_g + \frac{C_g}{E_0}) c t^2 \right)
 \right]\\
& \quad+
C M(\varepsilon) \left[ 1+ \frac{\eta M(\varepsilon)V(\varepsilon)}{\delta^2} + M(\varepsilon)\delta \right]\varepsilon (T V(\varepsilon))^2 + C (M(\varepsilon))^2 TV(\varepsilon) \left(\frac{1}{\eta } \right) \left( \frac{\varepsilon}{\eta}\right)^2
\\ &
\quad +
C_g  \frac{ M(\varepsilon) }{V(\varepsilon)^3} + \frac{M_{f_0}}{V(\varepsilon)^2} + \frac{ M(\varepsilon)}{V(\varepsilon)^2}
\left(1+ E_0^2 + (M_gE_0 +C_g)ct +\left(M_g+ \frac{C_g}{E_0}\right)^2 \frac{c^2 t^2}{2}\right).\numberthis \label{HOT}
\end{align*}
\end{theorem}
\begin{proof}
By using the proof of \cite[Lemma 5.7]{Matthies2018}, the previous Lemma \ref{2.4.3.} and Proposition \ref{2.3.1.}, the above estimate follows.
\end{proof}
\subsection{  Making time $T$ large  }\label{sec:H}
The aim of this subsection is to make time $T$ large, as $\varepsilon$ tends to be very small. We are doing this by equating the leading order terms of the bound of the probability of bad trees and making them small. Therefore we have the following lemma. 
\begin{lemma}
We can choose the parameter $c$ and the time $T$ in Theorem $\ref{ThP}$ to be of the form
\begin{align*}
c^{\frac{84}{103}}T = \varepsilon^{\frac{52\alpha}{103} -\frac{11}{103}}
\end{align*} which becomes large, as $\varepsilon$ becomes small, for the values of $\alpha$ between $ 0 < \alpha < \frac{11}{52}.$
\end{lemma}
\begin{proof}
In this step, we want to make the time $cT$ large such that $\lim_{ \varepsilon \to 0} cT(\varepsilon) = \infty.$
By considering only the leading order terms in \eqref{HOT} we have
 \begin{align}\label{balance}
\frac{M(\varepsilon) c^2 T^2}{V(\varepsilon)^2}
=
\frac{c^2T^2}{M(\varepsilon)}
=
\varepsilon M(\varepsilon)^2 T^2 V(\varepsilon)^2 \delta
=
\frac{\varepsilon M(\varepsilon)^2 V(\varepsilon)^3 \eta T^2 }{\delta^2}
=
\frac{M(\varepsilon)^2 TV(\varepsilon)\varepsilon^2}{\eta^3}.
\end{align}
We solve this system
\begin{align}\label{M}
\frac{M(\varepsilon) c^2 T^2}{V(\varepsilon)^2}
=
\frac{c^2T^2}{M(\varepsilon)}
&\Leftrightarrow
M(\varepsilon) = V(\varepsilon),\\
\frac{\varepsilon M(\varepsilon)^2 V(\varepsilon)^3 \eta T^2}{\delta^2} 
=
\frac{\varepsilon^2 T M(\varepsilon)^2 V(\varepsilon)}{\eta^3}
&\Leftrightarrow
\eta = \frac{\varepsilon^{\frac{1}{4}} \delta ^{\frac{1}{2}}}{V(\varepsilon)^{\frac{1}{2}} T^{\frac{1}{4}}},\label{eta}
\\
\frac{c^2T^2}{M(\varepsilon)}
=
\varepsilon M(\varepsilon)^2 T^2 V(\varepsilon)^2 \delta 
&\Leftrightarrow
\delta = \frac{c^2}{V(\varepsilon)^2 \varepsilon}.\label{delta}
\end{align}
Then by \eqref{M}, \eqref{eta} and \eqref{delta} we obtain
\begin{align}
\frac{c^2T^2}{M(\varepsilon)} 
=
\frac{M(\varepsilon)^2 T V(\varepsilon) \varepsilon^2}{\eta^3}
&\Leftrightarrow
c^2 T^2 \eta^3
=
M(\varepsilon)^3 T V(\varepsilon)\varepsilon^2 \nonumber
\\
&\Leftrightarrow
V(\varepsilon)
=
c^{\frac{5}{13}} T^{\frac{1}{52}} \varepsilon^{-\frac{11}{52}}. \label{V}
\end{align}
Now, we want to make all the terms in \eqref{balance} small. We observe that, by the choices of $M(\varepsilon)$, $V(\varepsilon)$, $\eta$ and $\delta$ as above, all the terms in \eqref{balance} are equal to $$\frac{c^2 T^2}{M(\varepsilon)}= \frac{c^2 T^2 }{V(\varepsilon)}=c^{\frac{21}{13}}T^{\frac{103}{52}} \varepsilon^{\frac{11}{52}}.$$
Thus we are equating this term with $\varepsilon^{\alpha}$ and we aim to find the values of this $\alpha$ in order for $c^{\frac{84}{103}}T$ to be large for small $\varepsilon$. That is
\begin{align} \label{1}
c^{\frac{21}{13}}T^{\frac{103}{52}} \varepsilon^{\frac{11}{52}} =
\varepsilon^{\alpha}   \Leftrightarrow
c^{\frac{84}{103}}T = \varepsilon^{\frac{52\alpha}{103} -\frac{11}{103}}.
\end{align}
Hence, the values of $\alpha$ that satisfy \eqref{1} are between
\begin{align*}
 \frac{52\alpha}{103} < \frac{11}{103}  
 \Leftrightarrow  
 0 < \alpha < \frac{11}{52}.
\end{align*}

\end{proof}
Now, we have all the tools we need to give the proof of the main theorem, Theorem \ref{thm:1.1.}.
\begin{proof}[Proof of Theorem \ref{thm:1.1.}] By the proof of \cite[Theorem 2.3]{Matthies2018} and Theorem \ref{ThP} above, we get that for any $t\in [0,T_{\varepsilon}]$ and for any measurable $\Omega \subset \mathds{T}^3 \times \mathds{R}^3$
 \begin{align*}
& \sup_{\Omega \subset \mathds{T}^3 \times \mathds{R}^3  } \big| \int_{\Omega} \hat{f}^N_t(x,v) - f_t (x,v) \, \d x \, \d v \big| =\sup_{\Omega \subset \mathds{T}^3 \times \mathds{R}^3  }|\hat{P}^{\varepsilon}_t (S_t(\Omega)) - P_t(S_t(\Omega)) | \\
&
\le
\frac{4}{3} \pi c \varepsilon
+C_1 C_2(1+M(\varepsilon)) T \varepsilon^2 (\beta + V(\varepsilon))^2 + M(\varepsilon) \varepsilon
+ C \varepsilon +  C M(\varepsilon)  \varepsilon^{3/2}\\ & \quad
+  \frac{1}{M(\varepsilon)} \left[
1+ c \pi \left( (\beta + E_0)t 
+ \frac{1}{2 \sqrt{2}} (M_g + \frac{C_g}{E_0}) c t^2 \right)
 \right]\\
& \quad+
C M(\varepsilon) \left[ 1+ \frac{\eta M(\varepsilon)V(\varepsilon)}{\delta^2} + M(\varepsilon)\delta \right]\varepsilon (T V(\varepsilon))^2 + C (M(\varepsilon))^2 TV(\varepsilon) \left(\frac{1}{\eta } \right) \left( \frac{\varepsilon}{\eta}\right)^2
\\ &
\quad +
C_g  \frac{ M(\varepsilon) }{V(\varepsilon)^3} + \frac{M_{f_0}}{V(\varepsilon)^2}  + \frac{ M(\varepsilon)}{V(\varepsilon)^2}
\left(1+ E_0^2 + (M_gE_0 +C_g)ct +\left(M_g+ \frac{C_g}{E_0}\right)^2 \frac{c^2 t^2}{2}\right)\\
& \le C \varepsilon ^{\alpha},
 \end{align*}
for the choices of $c^{\frac{84}{103}}T = \varepsilon^{\frac{52\alpha}{103} -\frac{11}{103}}$, $V(\varepsilon)= M(\varepsilon) = c^{\frac{5}{13}} T^{\frac{1}{52}} \varepsilon^{-\frac{11}{52}}$ and $\eta$, $\delta$ as in \eqref{eta}, \eqref{delta} with $ 0< \alpha < \frac{11}{52}$.
This concludes the proof of Theorem $\ref{thm:1.1.}$, since
for any $t\in [0,T_{\varepsilon}]$
\begin{align*}
\| \hat{f}^N_t - f_t \|_{L^1(\mathds{T}^3 \times \mathds{R}^3)}
&= 2 \int_{f_t^N \ge f_t}
\hat{f}^N_t (x,v) - f_t(x,v)\, \d x\, \d v \\
&\le
 \sup_{\Omega \subset \mathds{T}^3 \times \mathds{R}^3  }
\big| \int_{\Omega} \hat{f}^N_t(x,v) - f_t (x,v)\, \d x\, \d v \big| \le C \varepsilon ^{\alpha},
\end{align*}
as required.
\end{proof}

\section{From Linear Boltzmann to the  Heat equation}

In the next step, we rescale the derived linear Boltzmann equations to obtain a diffusion equation as a scaling limit. There are several suitable function spaces to study linear Boltzmann equations. In the previous section, we followed \cite{Matthies2018} and used $L^1$ theory as in \cite{arlotti06}. The use of $L^2$ theories have been useful for decay estimates, see e.g. \cite{Bisi2015,Canizio2018}.
By \cite[Section 4] {Gallagher2019} we know that, as soon as the initial data of the linear Boltzmann equation \eqref{LB} belongs to $L^{\infty}$, then there is a unique global solution to $\eqref{LB}$. Then, since $f_t(x,v)$ is a solution of the linear Boltzmann equation $\eqref{1.1}$, by uniqueness of solution we get that $f_t = \phi_t M_{\beta}$.

The limit from the mesoscopic to the macroscopic description consists in taking the hydrodynamic limit $c\to \infty$. In order to do so, we rescale the linear Boltzmann equation $\eqref{1.1}$ and take equation $\eqref{LB}$. In the macroscopic limit, the trajectory of the tagged particle is defined by $\Xi (\tau) := x(c \tau) \in \mathds{T}^3$. The distribution of $\Xi (\tau)$ is given by $\hat{f}^N(c \tau, x, v)$.
In \cite[Section 6]{MR3455156} they are connecting the linear Boltzmann equation to the heat equation in the following sense.
 Let $f_0$ the initial distribution of the tagged particle in a background distributed according to a Maxwellian $M_{\beta}$. Consider $\rho_0$ a continuous density of probability on $\mathds{T}^3$. Let also $\phi_{c}$ be the solution of the linear Boltzmann equation
\begin{align*}
\partial_t \phi_{c} + v \cdot \nabla_x \phi_{c} &= - c  \mathcal{L} \phi_{c}\\
\mathcal{L}\phi_{c} := \int \int [\phi_{c} (v) - \phi_{c}(v')]& M_{\beta}(v_1) [(v - v_1)\cdot \nu]_{+} \, \d v_1 \, \d\nu  \numberthis \label{LB} \\
M_{\beta}(v) := (\frac{\beta}{2\pi})^{\frac{3}{2}} &\exp (- \frac{\beta}{2} |v|^2), \quad \beta >0,
\end{align*}
with initial data $\rho_0$.
Then for all $\tau \in [0,T]$
\begin{align}\label{29.}
    \| M_{\beta}(v) [\phi_{c}(c \tau, x, v) - \rho(\tau, x)]\|_{L^{\infty} ([0,T] \times \mathds{T}^3 \times \mathds{R}^3)} \to 0 \quad \mathrm{as} \ c \to \infty,
\end{align}
 where $\rho(\tau,x)$ is the solution of the linear heat equation
\begin{align*}
 \partial_{\tau} \rho - \kappa_{\beta} \Delta _x \rho &=0 \quad \quad \mathrm{in}\ \mathds{T}^3,\\
 \rho (0, \cdot)&= \rho_0,
\end{align*}
where the diffusion parameter $\kappa_{\beta}$ is given by
\begin{equation*}
\kappa_{\beta} := \frac{1}{3} \int_{\mathds{R}^3}
v \mathcal{L}^{-1}v M_{\beta}(v)
\, \d v,
\end{equation*}
where $\mathcal{L}$ is the Linear Boltzmann operator and $\mathcal{L}^{-1}$ is its pseudo-inverse defined on $(\mathrm{Ker} \mathcal{L})^{ \perp}$.
 The factor $3$ in the diffusion parameter comes from the dimension $d=3$. See, \cite[Section 6]{MR3455156} for a proof of this convergence. Now, we are ready to give the proof of Theorem $\ref{thm2}$.
 \begin{proof}[Proof of Theorem \ref{thm2}]
By the triangle inequality we have
\begin{align*}
 \| \hat{f}^N (c t,  x,v) - \rho(\tau,x) M_{\beta}(v) \|_{L^1 ([0,T] \times \mathds{T}^3 \times \mathds{R}^3)}
 &\le
 \| \hat{f}^N (c t,  x,v) - f (c t,  x,v)\|_{L^1 ([0,T] \times \mathds{T}^3 \times \mathds{R}^3)}\\
 & \quad+
 \| f (c t,  x,v) - \rho(\tau,x) M_{\beta}(v) \|_{L^1 ([0,T] \times \mathds{T}^3 \times \mathds{R}^3)}\\
 & =
 \| \hat{f}^N (c t,  x,v) - M_{\beta}(v) \phi_c (c t,  x,v)\|_{L^1 ([0,T] \times \mathds{T}^3 \times \mathds{R}^3)}\\
 &\quad+
 \| M_{\beta}(v) [ \phi_c (c t,  x,v) - \rho(\tau,x)] \|_{L^1 ([0,T] \times \mathds{T}^3 \times \mathds{R}^3)},
\end{align*}
where by Theorem \ref{thm:1.1.} we know that
\begin{align*}
 \| \hat{f}^N (c t,  x,v) - M_{\beta}(v) \phi_c (c t,  x,v)\|_{L^1 ([0,T] \times \mathds{T}^3 \times \mathds{R}^3)} \to 0  \quad \mathrm{as} \ N \to \infty.
\end{align*}
Furthermore, by $\eqref{29.}$ we know that
\begin{align*}
    \| M_{\beta}(v) [\phi_{c}(c \tau, x, v) - \rho(\tau, x)]\|_{L^{\infty} ([0,T] \times \mathds{T}^3 \times \mathds{R}^3)} \to 0 \quad \mathrm{as} \ c \to \infty.
\end{align*}
Then
\begin{align*}
   & \| M_{\beta}(v) [\phi_{c}(c \tau, x, v) - \rho(\tau, x)]\|_{L^1 ([0,T] \times \mathds{T}^3 \times \mathds{R}^3)} \\
&=
\int_{0}^{T} \int_{\mathds{T}^3 \times \mathds{R}^3} |M_{\beta}(v) [\phi_{c}(c \tau, x, v) - \rho(\tau, x)]  | \, \d x\, \d v \, \d \tau \\
& =
\int_{0}^{T} \int_{\mathds{T}^3 \times B_R} |M_{\beta}(v) [\phi_{c}(c \tau, x, v) - \rho(\tau, x)]  | \, \d x\, \d v \, \d \tau \\ & \quad+
\int_{0}^{T} \int_{\mathds{T}^3 \times (\mathds{R}^3 \setminus B_R) } |M_{\beta}(v) [\phi_{c}(c \tau, x, v) - \rho(\tau, x)]  | \, \d x\, \d v \, \d \tau,
\end{align*}
where
\begin{align*}
\int_{0}^{T} \int_{\mathds{T}^3 \times B_R} |M_{\beta}(v) [\phi_{c}(c \tau, x, v) &- \rho(\tau, x)]  | \, \d x\, \d v \, \d \tau \\
&\le
\| M_{\beta}(v) [\phi_{c}(c \tau, x, v) - \rho(\tau, x) \|_{L^{\infty}} \int_{[0,T] \times \mathds{T}^3 \times \mathds{R}^3}
1 \, \d x\, \d v \, \d \tau \\
& \le
C \| M_{\beta}(v) [\phi_{c}(c \tau, x, v) - \rho(\tau, x) \|_{L^{\infty}}.\numberthis \label{29}
\end{align*}
and
\begin{align}\label{30}
\int_{0}^{T} \int_{\mathds{T}^3 \times (\mathds{R}^3 \setminus B_R)} |M_{\beta}(v) [\phi_{c}(c \tau, x, v) &- \rho(\tau, x)]  | \, \d x\, \d v \, \d \tau
< \delta,
\end{align}
by making $R$ large enough as $ M_{\beta} [ \phi_c - \rho]$ has velocity moments.
Thus, by equations \eqref{29} and \eqref{30} we get
\begin{align*}
    \| M_{\beta}(v) [\phi_{c}(c \tau, x, v) - \rho(\tau, x)]\|_{L^1 ([0,T] \times \mathds{T}^3 \times \mathds{R}^3)} \to 0 \quad \mathrm{as}\  c \to \infty.
    \end{align*}
Therefore
\begin{align*}
 \| \hat{f}^N (c t,  x,v) - \rho(\tau,x) M_{\beta}(v) \|_{L^1 ([0,T] \times \mathds{T}^3 \times \mathds{R}^3)} \to 0  \quad \mathrm{as} \ c \to \infty,
 \end{align*}
 which concludes the proof of Theorem $ \ref{thm2}$.

 \end{proof}

 \section*{Acknowledgements}
This work was supported through  The Leverhulme Trust  RPG-2020-107.




\end{document}